\documentclass[11pt,reqno,a4paper,twoside]{extarticle}

\usepackage[osf]{newtxtext}

\usepackage[left=3cm,right=3cm,top=2.2cm,bottom=2.2cm]{geometry}

\usepackage{graphicx, varioref, amsfonts}
\usepackage{amsthm,amssymb,amsmath,mathrsfs,mathtools,stmaryrd} 
\usepackage{dsfont}
\usepackage{upgreek} 
\usepackage{esint} 
\usepackage{breakurl}
\usepackage[full]{textcomp}
\usepackage{empheq}
\usepackage{enumerate,enumitem}
\usepackage{soul}
\usepackage{cancel}
\usepackage{amscd}
\usepackage{srcltx}
\usepackage{ifthen}
\usepackage{float}
\usepackage{color}
\usepackage{comment}
\usepackage{chngcntr}
\usepackage{etoolbox}
\usepackage{fancyhdr}
\usepackage{caption}

\usepackage{titlefoot}

\usepackage{authblk} 
\theoremstyle{plain}
\newtheorem{lemma}{Lemma}[section]
\newtheorem{proposition}{Proposition}[section]

\newtheorem{theorem}{Theorem}[section]
\newtheorem{assumption}{Assumption}[section]

\theoremstyle{definition}

\newtheorem{remark}{Remark}[section]

\newlist{todolist}{itemize}{2}
\setlist[todolist]{label=$\square$}
\usepackage{pifont}
\usepackage{hyperref}
\usepackage{cleveref}
\begin{document}

\title{Peng's Maximum Principle for Stochastic Partial Differential Equations}
\newcommand\shorttitle{Peng's Maximum Principle for SPDEs}
\date{May 11, 2021}

\author{Wilhelm Stannat}
\author{Lukas Wessels}
\newcommand\authors{Wilhelm Stannat and Lukas Wessels}

\affil{\small Technische Universit\"at Berlin}

\maketitle

\unmarkedfntext{\textit{Mathematics Subject Classification (2020) ---} Primary: 93E20, 49K45; Secondary: 60H15.}

\unmarkedfntext{\textit{Keywords and phrases ---} stochastic maximum principle, Pontryagin maximum principle, stochastic optimal control, necessary condition, spike variation.}

\unmarkedfntext{\textit{Email}: \textbullet$\,$ stannat@math.tu-berlin.de $\,$\textbullet$\,$ wessels@math.tu-berlin.de}


\begin{abstract}
  We extend Peng's maximum principle for semilinear stochastic partial differential equations (SPDEs) in one space-dimension with non-convex control domains and control-dependent diffusion coefficients to the case of general cost functionals with Nemytskii-type coefficients. Our analysis is based on a new approach to the characterization of the second order adjoint state as the solution of a function-valued backward SPDE. 
\end{abstract}


\section{Introduction}
In the 1950s, Pontryagin et al. established a necessary condition for the optimal control of finite-dimensional deterministic control problems, marking a milestone in mathematical control theory, see \cite{pontryagin1962}. Following their work, there have been many attempts to generalize this result to the optimal control of stochastic differential equations (SDEs). In his seminal paper \cite{peng1990}, Peng derived the corresponding necessary condition for the optimal control of SDEs in full generality. His work shows in particular that the necessary condition significantly differs from the one in the deterministic case, if the control domain is non-convex and the control appears in the diffusion term.

In contrast to the deterministic case, in the stochastic case, It\^o's calculus requires the consideration of Taylor expansions up to second order. In order to handle the quadratic terms arising in this expansion, a second order adjoint state must be introduced which, in the finite-dimensional case, solves a matrix-valued backward stochastic differential equation (BSDE). In infinite dimensions, handling the corresponding equation is the main obstacle for the generalization of the stochastic maximum principle for SPDEs. The main result of the present paper now is to give a new characterization of this process in terms of a function-valued backward SPDE. 

More specifically, we consider the following class of control problems of SPDEs: Minimize 
\begin{equation}\label{costfunctional}
J(u) := \mathbb{E} \left [ \int_0^T \int_{\Lambda} l(x^u_t(\lambda),u_t) \mathrm{d}\lambda \mathrm{d}t + \int_{\Lambda} h(x^u_T(\lambda)) \mathrm{d}\lambda \right ]
\end{equation}
subject to
\begin{equation}\label{stateequation}
\begin{cases}
\mathrm{d}x^u_t = [ \Delta x^u_t + b(x^u_t, u_t) ] \mathrm{d}t + \sigma(x^u_t,u_t) \mathrm{d}W_t &\text{on } [0,T]\times L^2(\Lambda)\\
x^u_0=x_0&\text{in }L^2(\Lambda).
\end{cases}
\end{equation}
Here, $\Lambda \subset \mathbb R$ is a bounded interval, $b$ and $\sigma$ are Nemytskii operators, and $l:\mathbb R\times U \to \mathbb{R}$ and $h:\mathbb R \to \mathbb{R}$ are functions of at most quadratic growth in $x\in \mathbb{R}$. The control domain $U$ is not assumed to be convex. Furthermore, $W$ is a cylindrical Wiener process on a Hilbert space $\Xi$. For precise assumptions, see \cref{notation}.

The second order adjoint state naturally takes values in the tensor product of the state space and can be characterized as the formal solution to an operator-valued BSDE. In the case of cost functionals of the above type \cref{costfunctional}, the state space of the associated BSDE is a non-separable Banach space, leading to further technical difficulties for its analysis. 

In the particular case of $L^2(\Lambda)$-valued SPDEs, we can use the identification $L^2(\Lambda) \otimes L^2(\Lambda) \cong L^2(\Lambda^2)$ and reformulate the associated BSDE as an equation on the function space $L^2(\Lambda^2)$, see equation \cref{generalsecondorderadjoint}. This characterization leads to additional structure on the adjoint state and enables us to generalize Peng's maximum principle to controlled semilinear SPDEs with at most quadratically growing costs.    

In this reformulation, the second order adjoint state satisfies a backward SPDE with distributional terminal condition. This type of terminal condition is not covered by the theory of backward SPDEs. Therefore, in the first step we mollify the terminal condition and obtain a solution to the mollified equation by known results. In the second step, we pass to the limit and obtain a solution to the general second order adjoint equation in the space
\begin{equation}
L^2([0,T]\times\Omega; L^2(\Lambda^2)) \cap L^2(\Omega; C([0,T];H^{-1}(\Lambda^2))).
\end{equation}
In subsequent work, we will exploit this characterization to relate the second order adjoint state to the second order viscosity differentials of the associated value function. We also expect that the representation of the second order adjoint state as a function will prove useful in further investigations of the underlying control problem, as well as numerical approximations of the second order adjoint state. 

Let us mention that the generalization to a uniformly elliptic operator with Dirichlet or Neumann boundary conditions is straightforward without any additional assumptions on the remaining coefficients. The necessary modifications are indicated in \Cref{changesoperator1,changesoperator2}. The restriction to one space-dimension can also be relaxed, if one assumes higher space-regularity on the noise coefficient $\sigma$, see \Cref{changesdimension}.

A large part of the existing literature on the generalization of Peng's maximum principle to infinite dimensions analyzes the operator-valued BSDE. This ansatz is pursued by Frankowska and Zhang in \cite{frankowska2020}, by L\"u and Zhang in \cite{lue2014, lue2015, lue2018}, and by Tang and Li in \cite{tang1994}. The drawback of this approach is that for nonlinear state equations one has to impose strong assumptions on the coefficients of the cost functional, excluding in particular quadratic costs. In \cite{guatteri2014}, Guatteri and Tessitore prove the existence and uniqueness of a solution to operator-valued BSDEs in the space of symmetric bounded operators. This class of equations includes in particular the second order adjoint equation for quadratic cost functionals. However, the proof of Peng's maximum principle requires a control of the terms arising in the duality relation as $\varepsilon$ tends to zero. This analysis has not been executed in the existing literature up to our knowledge.

Let us also mention a different approach to characterize the second order adjoint state, which is pursued by Du and Meng in \cite{du2013}, and by Fuhrman, Hu and Tessitore in \cite{fuhrman2012} and \cite{fuhrman2013}. In this approach, the second order adjoint state is characterized as a so-called stochastic bilinear form. In Du and Meng's work, the assumption that the coefficients of the state equation are twice Fr\'echet differentiable is too restrictive to include, for example, general Nemytskii-type operators. On the other hand, Fuhrman, Hu and Tessitore cannot include quadratic costs. Moreover, the second order adjoint state is not characterized as the solution to a BSDE as it is in the case in finite dimensions. For further references concerning the stochastic maximum principle see the survey article \cite{hu2019}. 

The rest of the paper is structured as follows. In \cref{notation}, we introduce some notation and state the assumptions for our control problem. In \cref{variationalinequatlity}, we recall Peng's strategy to derive the variational inequality in the finite-dimensional setting and discuss the modifications which have to be made in the infinite-dimensional case. In \cref{adjointprocesses}, we define the first order adjoint state and the mollified second order adjoint state via Riesz' representation theorem. We show how the formulation of the control problem with Nemytskii operators enables us to introduce the mollified second order adjoint state in the space $L^2(\Lambda^2)$. In \cref{adjointequations}, we prove that the first order adjoint state satisfies the first order adjoint equation. The mollified second order adjoint state is characterized via a mollified second order adjoint equation. In \cref{limit}, we pass to the limit of the mollified second order adjoint state and derive the second order adjoint equation. Finally, in \cref{mainresult}, we state and prove the main result of our work, the stochastic maximum principle for SPDEs.

\section{Notation and Assumptions}\label{notation}
\noindent
Let $\Lambda \subset \mathbb{R}$ be a bounded interval. For $\gamma>0$, let $H^{\gamma}_0(\Lambda) := W^{\gamma,2}_0$ be the fractional Sobolev space of order $\gamma$ with Dirichlet boundary conditions, and let $H^{-\gamma}(\Lambda)$ denote its dual space. We fix a finite time horizon $T>0$, and solve the state equation \cref{stateequation} in the variational setting, i.e. we work on the Gelfand triple
\begin{equation}
H^1_0(\Lambda) \hookrightarrow L^2(\Lambda) \hookrightarrow H^{-1}(\Lambda),
\end{equation}
and realize $\Delta: H^1_0(\Lambda)\to H^{-1}(\Lambda)$ as a continuous operator (for details see \cite{liu2015}).

For real, separable Hilbert spaces $\mathcal X$, $\mathcal Y$, we denote by $L_2(\mathcal X, \mathcal Y)$ the space of Hilbert-Schmidt operators from $\mathcal X$ to $\mathcal Y$, and by $L_1(\mathcal X, \mathcal Y)$ the space of trace-class operators from $\mathcal X$ to $\mathcal Y$. Furthermore, we denote $L_2(\mathcal X) := L_2(\mathcal X, \mathcal X)$ and $L_1(\mathcal X) := L_1(\mathcal X, \mathcal X)$.

We make the following assumptions.
\begin{assumption}
	\begin{enumerate}[start=1,label={(A\arabic*)}]
		\item\label{assumptionwienerprocess} Let $\Xi$ be a separable Hilbert space and let $(W_t)_{t\in [0,T]}$ be a $\Xi$-valued cylindrical Wiener process on a filtered probability space\\ $(\Omega, \mathcal F, (\mathcal F_t)_{t\in [0,T]}, \mathbb{P})$.
		\item\label{assumtpioncontrolspace} Let $U$ be a non-empty subset of a separable Banach space $\mathcal{U}$, and let
		\begin{multline}\label{admissiblecontrols}
		U_{\text{ad}} := \Big \{ u: [0,T]\times\Omega \to U : \; u\; ( \mathcal F_t )_{t\in [0,T]}-\text{adapted and}\\
		\sup_{t\in [0,T]} \mathbb{E} \left [ \| u_t \|_{\mathcal{U}}^k \right ] < \infty, \forall k\in \mathbb{N} \Big \}
		\end{multline}
		be the set of admissible controls. In particular, $U_{\text{ad}}$ is not assumed to be convex.
		\item\label{assumptioncoefficients} Let $b:\mathbb{R} \times U \to \mathbb{R}$, $\sigma: \mathbb{R} \times U \to L_2(\Xi,\mathbb{R})$, $l:\mathbb{R} \times U \to \mathbb{R}$, and $h: \mathbb{R} \to \mathbb{R}$ be continuous in both arguments and twice continuously differentiable with respect to the first argument. We denote this derivative by $b_{x}$, etc. Let $b_{x}$, $b_{xx}$, $\sigma_{x}$, $\sigma_{xx}$, $l_{xx}$, $h_{xx}$ be bounded, and $b$, $\sigma$, $l_{x}$, $h_{x}$ be bounded by $(1+|x| +\|u\|_{\mathcal{U}})$, $x\in \mathbb{R}$, $u\in U$.
		\item Let $x_0 \in L^2(\Lambda)$.
	\end{enumerate}
\end{assumption}

\begin{remark}\label{changesoperator1}
	\begin{enumerate}
		\item All these coefficients give rise to Nemytskii operators on $L^2(\Lambda)$. For example, we have an operator
		\begin{equation}
		\begin{split}
		\sigma: &L^2(\Lambda)\times U \to L_2(\Xi,L^2(\Lambda)),\\
		&(x,u) \mapsto \left ((\xi,\lambda) \mapsto  \sigma(x(\lambda),u)(\xi) \right ).
		\end{split}
		\end{equation}
		Throughout the paper, we are going to use the identification
		\begin{equation}
		\begin{split}
		L^2(\Lambda; L_2(\Xi,\mathbb{R}))&\cong L_2(\Xi,L^2(\Lambda))\\
		q(\lambda)(\xi) &\leftrightarrow q(\xi)(\lambda).
		\end{split}
		\end{equation}
		\item The differential operator $\Delta$ in the state equation \cref{stateequation} can be replaced by the generator $A: \mathcal{D}(A) \subset L^2(\Lambda) \to L^2(\Lambda)$ of the quadratic form
		\begin{equation}
		\int_{\Lambda} a(\lambda) (\partial_{\lambda} x)^2(\lambda) \mathrm{d}\lambda, \quad x\in H^1_0(\Lambda),
		\end{equation}
		for some $a\in L^{\infty}(\Lambda)$ with $a(\lambda) \geq a_0 >0$, that can be formally represented as the second order differential operator in divergence form
		\begin{equation}\label{generaloperator}
		Ax(\lambda) = \partial_{\lambda}(a \partial_{\lambda} x)(\lambda).
		\end{equation}
	\end{enumerate}
\end{remark}

\section{Variational Inequality}\label{variationalinequatlity}
\noindent
We follow the idea of the finite-dimensional case and introduce a so-called spike variation. Let $\bar u$ be an optimal control of the control problem \cref{costfunctional} and \cref{stateequation}, and let $\bar x$ be the associated optimal state. Fix any $v\in U$, $\tau \in (0,T)$ and $\varepsilon >0$, and define
\begin{equation}
u^{\varepsilon}_t := \begin{cases}
v , &\tau\leq t \leq \tau+\varepsilon\\
\bar u_t, &\text{otherwise}.
\end{cases}
\end{equation}
Let $x^{\varepsilon}$ denote the state associated with $u^{\varepsilon}$. Let $y^{\varepsilon}$ denote the solution to the SPDE
\begin{equation}\label{yepsilon}
\begin{cases}
\mathrm{d}y^{\varepsilon}_t = \left [ \Delta y^{\varepsilon}_t + b_{x}(\bar x_t, \bar u_t) y^{\varepsilon}_t + b(\bar x_t, u^{\varepsilon}_t) - b(\bar x_t, \bar u_t) \right ] \mathrm{d}t\\
\qquad\qquad + \left [ \sigma_{x}(\bar x_t, \bar u_t) y^{\varepsilon}_t + \sigma(\bar x_t, u^{\varepsilon}_t) - \sigma(\bar x_t, \bar u_t) \right ] \mathrm{d}W_t\\
y^{\varepsilon}_0 = 0,
\end{cases}
\end{equation}
and let $z^{\varepsilon}$ denote the solution to the SPDE
\begin{equation}\label{zepsilon}
\begin{cases}
\mathrm{d}z^{\varepsilon}_t = \left [\Delta z^{\varepsilon}_t + b_{x}(\bar x_t, \bar u_t) z^{\varepsilon}_t + \frac12 b_{xx}(\bar x_t, \bar u_t)y^{\varepsilon}_t y^{\varepsilon}_t + (b_{x}(\bar x_t, u^{\varepsilon}_t) - b_{x}(\bar x_t, \bar u_t)) y^{\varepsilon}_t \right ] \mathrm{d}t\\
\qquad\qquad+ \left [ \sigma_{x}(\bar x_t, \bar u_t) z^{\varepsilon}_t + \frac12 \sigma_{xx}(\bar x_t, \bar u_t)y^{\varepsilon}_t y^{\varepsilon}_t + (\sigma_{x}(\bar x_t, u^{\varepsilon}_t) - \sigma_{x}(\bar x_t, \bar u_t)) y^{\varepsilon}_t \right ] \mathrm{d}W_t\\
z^{\varepsilon}_0=0.
\end{cases}
\end{equation}
These equations are called first and second order variational equations, respectively.
\begin{lemma}\label{approximation}
	It holds
	\begin{equation}
	\sup_{t\in [0,T]} \mathbb{E} \left [ \left \| x^{\varepsilon}_t - \bar x_t - y^{\varepsilon}_t - z^{\varepsilon}_t \right \|_{L^2(\Lambda)}^2 \right ] \leq o(\varepsilon^2).
	\end{equation}
\end{lemma}

Before we prove this Lemma, we need the following Lemma as a preparation.
\begin{lemma}\label{bounds}
	It holds
	\begin{align}
	&\sup_{t\in [0,T]} \mathbb{E} \left [ \left \| y^{\varepsilon}_t \right \|_{L^2(\Lambda)}^{2k} \right ] \leq C\varepsilon^k\\
	&\sup_{t\in [0,T]} \mathbb{E} \left [ \left \| z^{\varepsilon}_t \right \|_{L^2(\Lambda)}^k \right ] \leq C\varepsilon^k,
	\end{align}
	for $k\in \mathbb{N}$.
\end{lemma}

\begin{remark}
In \Cref{subsequence} below, we prove higher space-regularity for $y^{\varepsilon}_T$.
\end{remark}

\begin{proof}
Let us begin with the inequalities for $y^{\varepsilon}$. By It\^o's formula for variational solutions to SPDEs (see \cite{liu2015}, Theorem 4.2.5) and elementary estimates, we have
\begin{equation}\label{itoy}
\begin{split}
&\| y^{\varepsilon}_t \|_{L^2(\Lambda)}^2\\
\leq& 2\int_0^t ( \|b_{x}\|_{\infty} + \| \sigma_{x} \|_{\infty}^2 + 1 ) \| y^{\varepsilon}_s \|_{L^2(\Lambda)}^2 \mathrm{d}s\\
&+2 \int_0^t \| b(\bar x_s, u^{\varepsilon}_s) - b(\bar x_s, \bar u_s) \|^2_{L^2(\Lambda)} + \| \sigma(\bar x_s, u^{\varepsilon}_s) - \sigma(\bar x_s, \bar u_s) \|_{L_2(\Xi,L^2(\Lambda))}^2 \mathrm{d}s\\
&+2\int_0^t \langle y^{\varepsilon}_s, \sigma_{x}(\bar x_s, \bar u_s) y^{\varepsilon}_s + \sigma(\bar x_s, u^{\varepsilon}_s) - \sigma(\bar x_s, \bar u_s) \mathrm{d}W_s \rangle_{L^2(\Lambda)}.
\end{split}
\end{equation}
Taking both sides to the power $k\in \mathbb{N}$, and taking the supremum and expectations, we arrive at
\begin{equation}\label{itoy2}
\begin{split}
&\mathbb{E} \left [ \sup_{t\in[0,T]} \| y^{\varepsilon}_t \|_{L^2(\Lambda)}^{2k} \right ]\\
\leq& C \int_0^T ( \|b_{x}\|_{\infty} + \| \sigma_{x} \|_{\infty}^2 + 1 )^k \mathbb{E} \left [ \sup_{r\in[0,s]} \| y^{\varepsilon}_r \|_{L^2(\Lambda)}^{2k} \right ] \mathrm{d}s\\
&+C \mathbb{E} \Bigg [ \Bigg ( \int_0^T \| b(\bar x_s, u^{\varepsilon}_s) - b(\bar x_s, \bar u_s) \|^2_{L^2(\Lambda)}\\
&\qquad \qquad\qquad\qquad + \| \sigma(\bar x_s, u^{\varepsilon}_s) - \sigma(\bar x_s, \bar u_s) \|_{L_2(\Xi,L^2(\Lambda))}^2 \mathrm{d}s \Bigg )^k \Bigg ]\\
&+C \mathbb{E} \left [ \sup_{t\in [0,T]} \left | \int_0^t \langle y^{\varepsilon}_s, \sigma_{x}(\bar x_s, \bar u_s) y^{\varepsilon}_s + \sigma(\bar x_s, u^{\varepsilon}_s) - \sigma(\bar x_s, \bar u_s) \mathrm{d}W_s \rangle_{L^2(\Lambda)} \right |^k \right ].
\end{split}
\end{equation}
Using Burkholder-Davis-Gundy inequality, we obtain
\begin{equation}
\begin{split}
&\mathbb{E} \left [ \sup_{t\in [0,T]} \left | \int_0^t \langle y^{\varepsilon}_s, \sigma_{x}(\bar x_s, \bar u_s) y^{\varepsilon}_s + \sigma(\bar x_s, u^{\varepsilon}_s) - \sigma(\bar x_s, \bar u_s) \mathrm{d}W_s \rangle_{L^2(\Lambda)} \right |^k \right ]\\
\leq &C \mathbb{E} \left [ \left \langle \int_0^{\cdot} \langle y^{\varepsilon}_s, \sigma_{x}(\bar x_s, \bar u_s) y^{\varepsilon}_s + \sigma(\bar x_s, u^{\varepsilon}_s) - \sigma(\bar x_s, \bar u_s) \mathrm{d}W_s \rangle_{L^2(\Lambda)} \right \rangle_T^{\frac{k}{2}} \right ]\\
\leq & C\mathbb{E} \left [ \sup_{t\in [0,T]} \| y^{\varepsilon}_t\|_{L^2(\Lambda)}^k \left ( \int_0^T \| \sigma_{x}(\bar x_s, \bar u_s) y^{\varepsilon}_s + \sigma(\bar x_s, u^{\varepsilon}_s) - \sigma(\bar x_s, \bar u_s) \|_{L_2(\Xi,L^2(\Lambda))}^2 \mathrm{d}s \right )^{\frac{k}{2}} \right ]\\
\leq & C\mathbb{E} \Bigg [ \frac{\alpha}{2} \sup_{t\in [0,T]} \| y^{\varepsilon}_t\|_{L^2(\Lambda)}^{2k}\\
&\qquad + \frac{1}{2\alpha} \left ( \int_0^T \| \sigma_{x}(\bar x_s, \bar u_s) y^{\varepsilon}_s + \sigma(\bar x_s, u^{\varepsilon}_s) - \sigma(\bar x_s, \bar u_s) \|_{L_2(\Xi,L^2(\Lambda))}^2 \mathrm{d}s \right )^{k} \Bigg ],
\end{split}
\end{equation}
for every $\alpha >0$. Choosing $\alpha >0$ sufficiently small, we derive from equation \cref{itoy2}
\begin{equation}
\begin{split}
&\mathbb{E} \left [ \sup_{t\in[0,T]} \| y^{\varepsilon}_t \|_{L^2(\Lambda)}^{2k} \right ]\\
\leq& C \int_0^T  \mathbb{E} \left [ \sup_{r\in[0,s]} \| y^{\varepsilon}_r \|_{L^2(\Lambda)}^{2k} \right ] \mathrm{d}s\\
&+C \mathbb{E} \! \left [ \! \left ( \int_0^T \! \| b(\bar x_s, u^{\varepsilon}_s) - b(\bar x_s, \bar u_s) \|^2_{L^2(\Lambda)} + \| \sigma(\bar x_s, u^{\varepsilon}_s) - \sigma(\bar x_s, \bar u_s) \|_{L_2(\Xi,L^2(\Lambda))}^2 \mathrm{d}s\! \right )^k \right ].
\end{split}
\end{equation}
Using the properties of $b$, we have
\begin{equation}
\begin{split}
&\mathbb{E} \left [ \left ( \int_0^T \| b(\bar x_s, u^{\varepsilon}_s) - b(\bar x_s, \bar u_s) \|_{L^2(\Lambda)}^2 \mathrm{d}s \right )^k \right ]\\
\leq & C \varepsilon^k \mathbb{E} \left [ 1+ \sup_{t\in[0,T]} \| \bar x_t \|_{L^2(\Lambda)}^{2k} + \sup_{t\in[0,T]} \| \bar u_t \|_{\mathcal{U}}^{2k} + \| v\|_{\mathcal{U}}^{2k} \right ],
\end{split}
\end{equation}
where the right-hand side is finite using a-priori estimates for variational solutions to SPDEs, see \cite{liu2015}, Theorem 5.1.3.	Analogously, we obtain the same estimate for the term involving $\sigma$. Gr\"onwall's inequality yields the claim for $y^{\varepsilon}$. 

The inequalities for $z^{\varepsilon}$ follow in a similar fashion. The higher order of convergence follows from the fact that the second order expansions in the equation for $z^{\varepsilon}$ satisfy twice the order of the convergence rates of the respective terms in the equation for $y^{\varepsilon}$.
\end{proof}

Now we come to the proof of \cref{approximation}.

\begin{proof}
Using Taylor's theorem for the G\^ateaux derivative (see \cite{zeidler1986}, Section 4.6) 
and the estimates from \cref{bounds}, the proof is exactly the same as in the finite-dimensional case (see \cite{peng1990}).
\end{proof}

With this result, we can derive the following inequality from the fact that $J(\bar u) \leq J(u^{\varepsilon})$. This inequality is the basis for deriving the variational inequality.
\begin{lemma}\label{inequalityp}
It holds
\begin{equation}
\begin{split}
&\mathbb{E} \left [ \int_0^T \int_{\Lambda} l_{x}(\bar x_t(\lambda), \bar u_t) (y^{\varepsilon}_t(\lambda) + z^{\varepsilon}_t(\lambda) ) + \frac12 l_{xx}(\bar x_t(\lambda), \bar u_t ) y^{\varepsilon}_t(\lambda) y^{\varepsilon}_t(\lambda) \mathrm{d}\lambda \mathrm{d}t \right ]\\
+ &\mathbb{E} \left [ \int_{\Lambda} h_{x} (\bar x_T(\lambda)) (y^{\varepsilon}_T(\lambda) + z^{\varepsilon}_T(\lambda)) + \frac12 h_{xx}(\bar x_T(\lambda))y^{\varepsilon}_T(\lambda) y^{\varepsilon}_T(\lambda) \mathrm{d}\lambda \right ]\\
+ &\mathbb{E} \left [ \int_0^T \int_{\Lambda} l(\bar x_t(\lambda), u^{\varepsilon}_t ) - l(\bar x_t(\lambda) , \bar u_t) \mathrm{d}\lambda \mathrm{d}t \right ] \geq o(\varepsilon).
\end{split}
\end{equation}
\end{lemma}

\begin{proof}
Again, similar to the proof in the finite-dimensional case (see \cite{peng1990}).
\end{proof}

\section{Adjoint States}\label{adjointprocesses}
\noindent
In this section, we are going to define the adjoint states using Riesz' representation theorem. We start with the first order adjoint state.

\subsection{First Order Adjoint State}
\noindent
Consider the SPDE
\begin{equation}\label{riesz1p}
\begin{cases}
\mathrm{d}y_t = \left [ \Delta y_t + b_{x}(\bar x_t, \bar u_t) y_t + \varphi_t \right ] \mathrm{d}t + \left [ \sigma_{x}(\bar x_t, \bar u_t ) y_t + \psi_t \right ] \mathrm{d}W_t\\
y_0 = 0,
\end{cases}
\end{equation}
where $(\varphi, \psi) \in L^2([0,T]\times \Omega ; L^2(\Lambda)) \times L^2([0,T]\times\Omega; L_2(\Xi, L^2(\Lambda)))$. Now, we construct a linear functional on the space $ L^2([0,T]\times\Omega; L^2(\Lambda)) \times L^2([0,T]\times\Omega; L_2(\Xi, L^2(\Lambda)))$ as follows
\begin{equation}
\mathcal T_1(\varphi, \psi) := \mathbb{E} \left [ \int_0^T \int_{\Lambda} l_{x}(\bar x_t(\lambda), \bar u_t) y_t(\lambda) \mathrm{d}\lambda \mathrm{d}t + \int_{\Lambda} h_{x}(\bar x_T(\lambda)) y_T(\lambda) \mathrm{d}\lambda \right ],
\end{equation}
where $y$ denotes the solution to equation \cref{riesz1p} associated with $(\varphi, \psi)$. By Riesz' representation theorem, there is a unique pair of adapted processes
\begin{equation}
(p,q)\in L^2([0,T]\times\Omega; L^2(\Lambda)) \times L^2([0,T]\times\Omega; L_2(\Xi, L^2(\Lambda))),
\end{equation}
such that
\begin{equation}\label{firstadjointstateproperty}
\mathcal T_1(\varphi, \psi) = \mathbb{E} \left [ \int_0^T \langle \varphi_t , p_t \rangle_{L^2(\Lambda)} + \langle \psi_t , q_t \rangle_{L_2(\Xi, L^2(\Lambda))} \mathrm{d}t \right ],
\end{equation}
for all $(\varphi, \psi) \in L^2([0,T]\times\Omega ; L^2(\Lambda)) \times L^2([0,T]\times\Omega; L_2(\Xi,L^2(\Lambda)))$. Equation \cref{firstadjointstateproperty} is called the first order adjoint state property. We exploit this property once for the process $y^{\varepsilon}$ given by equation \cref{yepsilon} and once for the process $z^{\varepsilon}$ given by equation \cref{zepsilon}. By choosing $(\varphi,\psi)$ accordingly, we can simplify the inequality from \cref{inequalityp} and obtain
\begin{equation}\label{step1p}
\begin{split}
&\mathbb{E} \left [ \int_0^T \int_{\Lambda} l_{x}(\bar x_t(\lambda), \bar u_t) (y^{\varepsilon}_t(\lambda) + z^{\varepsilon}_t(\lambda) ) + \frac12 l_{xx}(\bar x_t(\lambda), \bar u_t ) y^{\varepsilon}_t(\lambda) y^{\varepsilon}_t(\lambda) \mathrm{d}\lambda \mathrm{d}t \right ]\\
&+ \mathbb{E} \left [ \int_{\Lambda} h_{x} (\bar x_T(\lambda)) (y^{\varepsilon}_T(\lambda) + z^{\varepsilon}_T(\lambda)) + \frac12 h_{xx}(\bar x_T(\lambda)) y^{\varepsilon}_T(\lambda) y^{\varepsilon}_T(\lambda) \mathrm{d}\lambda \right ]\\
&+ \mathbb{E} \left [ \int_0^T \int_{\Lambda} l(\bar x_t(\lambda), u^{\varepsilon}_t ) - l(\bar x_t(\lambda) , \bar u_t) \mathrm{d}\lambda \mathrm{d}t \right ]\\
=& \mathbb{E} \left [ \int_0^T \left \langle b(\bar x_t, u^{\varepsilon}_t) - b(\bar x_t, \bar u_t) , p_t \right \rangle_{L^2(\Lambda)} + \left \langle \sigma(\bar x_t, u^{\varepsilon}_t) - \sigma( \bar x_t, \bar u_t) , q_t \right \rangle_{L_2(\Xi,L^2(\Lambda))} \mathrm{d}t \right ]\\
&+ \mathbb{E} \left [ \int_0^T \frac12 \left \langle b_{xx}( \bar x_t, \bar u_t)y^{\varepsilon}_t y^{\varepsilon}_t , p_t \right \rangle_{L^2(\Lambda)} + \frac12 \left \langle \sigma_{xx}( \bar x_t, \bar u_t)y^{\varepsilon}_t y^{\varepsilon}_t , q_t \right \rangle_{L_2(\Xi,L^2(\Lambda))} \mathrm{d}t \right ]\\
&+ \mathbb{E} \left [ \int_0^T \int_{\Lambda} \frac12 l_{xx}(\bar x_t(\lambda), \bar u_t) y^{\varepsilon}_t(\lambda) y^{\varepsilon}_t(\lambda) \mathrm{d}\lambda + \int_{\Lambda} l(\bar x_t(\lambda), u^{\varepsilon}_t) - l( \bar x_t(\lambda), \bar u_t) \mathrm{d}\lambda \mathrm{d}t \right ]\\
&+ \mathbb{E} \left [ \frac12 \int_{\Lambda} h_{xx}( \bar x_T(\lambda)) y^{\varepsilon}_T(\lambda) y^{\varepsilon}_T(\lambda) \mathrm{d}\lambda \right ] + o(\varepsilon).
\end{split}
\end{equation}
Note that the term
\begin{equation}
\begin{split}
&\mathbb{E} \Bigg [ \int_0^T \left \langle ( b_{x}( \bar x_t , u^{\varepsilon}_t) - b_{x}( \bar x_t, \bar u_t )) y^{\varepsilon}_t, p_t \right \rangle_{L^2(\Lambda)}\\
&\qquad\qquad+ \left \langle (\sigma_{x}( \bar x_t , u^{\varepsilon}_t) - \sigma_{x}(\bar x_t, \bar u_t )) y^{\varepsilon}_t , q_t \right \rangle_{L_2(\Xi,L^2(\Lambda))} \mathrm{d}t \Bigg ]
\end{split}
\end{equation}
is of order $o(\varepsilon)$ and hence can be omitted.

\subsection{Mollified Second Order Adjoint State}

In order to handle the quadratic terms using the same idea as for the linear terms, we have to turn the bilinear forms into linear forms on the tensor product $L^2(\Lambda)\otimes L^2(\Lambda) \cong L^2(\Lambda^2)$ (see \cite{reed1980}, Theorem II.10, for the isomorphism).

\begin{proposition}\label{tensorprocess}
The process $Y^{\varepsilon}_t(\lambda,\mu):= y^{\varepsilon}_t(\lambda) y^{\varepsilon}_t(\mu)$, $\lambda,\mu\in\Lambda$, is in the space
\begin{equation}
L^2([0,T]\times\Omega; H^1_0(\Lambda^2)) \cap L^2(\Omega;C([0,T];L^2(\Lambda^2)))
\end{equation}
and satisfies the equation
\begin{equation}\label{secondvariationalequation}
\begin{cases}
\mathrm{d}Y^{\varepsilon}_t(\lambda,\mu) = [\Delta Y^{\varepsilon}_t(\lambda,\mu) + ( b_{x}(\bar x_t(\lambda),\bar u_t) + b_{x}(\bar x_t(\mu),\bar u_t) ) Y^{\varepsilon}_t(\lambda,\mu)\\
\qquad\qquad\qquad + \langle \sigma_{x}(\bar x_t(\lambda),\bar u_t), \sigma_{x}(\bar x_t(\mu),\bar u_t) \rangle_{L_2(\Xi,\mathbb{R})} Y^{\varepsilon}_t(\lambda,\mu) + \Phi^{\varepsilon}_t(\lambda,\mu) ]\mathrm{d}t\\
\qquad\qquad\qquad + [ ( \sigma_{x}(\bar x_t(\lambda),\bar u_t) + \sigma_{x}(\bar x_t(\mu),\bar u_t) ) Y^{\varepsilon}_t(\lambda,\mu) + \Psi^{\varepsilon}_t(\lambda,\mu) ] \mathrm{d}W_t\\
Y^{\varepsilon}_0 =0,
\end{cases}
\end{equation}
where
\begin{equation}
(\Phi^{\varepsilon},\Psi^{\varepsilon}) \in L^2([0,T]\times\Omega; L^2(\Lambda^2)) \times L^2([0,T]\times \Omega; L_2(\Xi,L^2(\Lambda^2)))
\end{equation}
are given by
\begin{equation}\label{phi}
\begin{split}
\Phi^{\varepsilon}_t(\lambda,\mu) =& y^{\varepsilon}_t(\lambda) (b(\bar x_t(\mu),u^{\varepsilon}_t) - b(\bar x_t(\mu), \bar u_t)) + y^{\varepsilon}_t(\mu) (b(\bar x_t(\lambda),u^{\varepsilon}_t) - b(\bar x_t(\lambda), \bar u_t))\\
&+ \langle \sigma_{x}(\bar x_t(\lambda),\bar u_t) y^{\varepsilon}_t(\lambda), (\sigma(\bar x_t(\mu),u^{\varepsilon}_t) - \sigma(\bar x_t(\mu), \bar u_t))\rangle_{L_2(\Xi,\mathbb{R})}\\
&+ \langle \sigma_{x}(\bar x_t(\mu),\bar u_t) y^{\varepsilon}_t(\mu), (\sigma(\bar x_t(\lambda),u^{\varepsilon}_t) - \sigma(\bar x_t(\lambda), \bar u_t))\rangle_{L_2(\Xi,\mathbb{R})}\\
&+ \langle (\sigma(\bar x_t(\lambda),u^{\varepsilon}_t) - \sigma(\bar x_t(\lambda), \bar u_t)), (\sigma(\bar x_t(\mu),u^{\varepsilon}_t) - \sigma(\bar x_t(\mu), \bar u_t))\rangle_{L_2(\Xi,\mathbb{R})},
\end{split}
\end{equation}
and
\begin{equation}\label{psi}
\Psi^{\varepsilon}_t(\lambda,\mu) = (\sigma(\bar x_t(\lambda),u^{\varepsilon}_t) - \sigma(\bar x_t(\lambda), \bar u_t)) y^{\varepsilon}_t(\mu) + (\sigma(\bar x_t(\mu),u^{\varepsilon}_t) - \sigma(\bar x_t(\mu), \bar u_t)) y^{\varepsilon}_t(\lambda).
\end{equation}
\end{proposition}

\begin{proof}
The regularity of $Y^{\varepsilon}$ follows from the regularity of $y^{\varepsilon}$. Applying It\^o's product rule for real-valued semimartingales to
\begin{equation}
\langle Y^{\varepsilon}_t, f_1\otimes f_2 \rangle_{L^2(\Lambda^2)}  = \langle y^{\varepsilon}_t , f_1 \rangle_{L^2(\Lambda)} \langle y^{\varepsilon}_t , f_2 \rangle_{L^2(\Lambda)},
\end{equation}
$f_1,f_2\in H^1_0(\Lambda)$, and using a density argument yields 
\begin{equation}
\begin{split}
\mathrm{d}Y^{\varepsilon}_t(\lambda,\mu)=& y^{\varepsilon}_t(\lambda) \mathrm{d}y^{\varepsilon}_t(\mu) + y^{\varepsilon}_t(\mu) \mathrm{d}y^{\varepsilon}_t(\lambda) + \mathrm{d} \langle y^{\varepsilon}(\lambda), y^{\varepsilon}(\mu) \rangle_t\\
=& y^{\varepsilon}_t(\lambda) (\Delta_{\mu} y^{\varepsilon}_t(\mu) + b_{x}(\bar x_t(\mu),\bar u_t) y^{\varepsilon}_t(\mu) + b(\bar x_t(\mu),u^{\varepsilon}_t) - b(\bar x_t(\mu), \bar u_t) ) \mathrm{d}t\\
&+ y^{\varepsilon}_t(\mu) (\Delta_{\lambda} y^{\varepsilon}_t(\lambda) + b_{x}(\bar x_t(\lambda),\bar u_t) y^{\varepsilon}_t(\lambda) + b(\bar x_t(\lambda),u^{\varepsilon}_t) - b(\bar x_t(\lambda), \bar u_t)) \mathrm{d}t\\
&+ \langle (\sigma_{x}(\bar x_t(\lambda),\bar u_t) y^{\varepsilon}_t(\lambda) + \sigma(\bar x_t(\lambda),u^{\varepsilon}_t) - \sigma(\bar x_t(\lambda), \bar u_t)),\\
&\qquad\qquad ( \sigma_{x}(\bar x_t(\mu),\bar u_t) y^{\varepsilon}_t(\mu) + \sigma(\bar x_t(\mu),u^{\varepsilon}_t) - \sigma(\bar x_t(\mu), \bar u_t)) \rangle_{L_2(\Xi,\mathbb{R})} \mathrm{d}t\\
&+ y^{\varepsilon}_t(\lambda) ( \sigma_{x}(\bar x_t(\mu),\bar u_t) y^{\varepsilon}_t(\mu) + \sigma(\bar x_t(\mu),u^{\varepsilon}_t) - \sigma(\bar x_t(\mu), \bar u_t) ) \mathrm{d}W_t\\
&+ y^{\varepsilon}_t(\mu) ( \sigma_{x}(\bar x_t(\lambda),\bar u_t) y^{\varepsilon}_t(\lambda) + \sigma(\bar x_t(\lambda),u^{\varepsilon}_t) - \sigma(\bar x_t(\lambda), \bar u_t) ) \mathrm{d}W_t 
\end{split}
\end{equation}
in $L^2 (\Lambda^2)$. Note that
\begin{equation}
y^{\varepsilon}_t(\lambda) \Delta_{\mu} y^{\varepsilon}_t(\mu) + y^{\varepsilon}_t(\mu) \Delta_{\lambda} y^{\varepsilon}_t(\lambda) = \Delta Y^{\varepsilon}_t(\lambda,\mu).
\end{equation}
Combining the remaining terms in a similar fashion, we end up with
\begin{equation}
\begin{split}
\mathrm{d}Y^{\varepsilon}_t(\lambda,\mu) =& [\Delta Y^{\varepsilon}_t(\lambda,\mu) + ( b_{x}(\bar x_t(\lambda),\bar u_t) + b_{x}(\bar x_t(\mu),\bar u_t) ) Y^{\varepsilon}_t(\lambda,\mu)\\
& + \langle \sigma_{x}(\bar x_t(\lambda),\bar u_t), \sigma_{x}(\bar x_t(\mu),\bar u_t) \rangle_{L_2(\Xi,\mathbb{R})} Y^{\varepsilon}_t(\lambda,\mu) + \Phi^{\varepsilon}_t(\lambda,\mu) ]\mathrm{d}t\\
&+ [ ( \sigma_{x}(\bar x_t(\lambda),\bar u_t) + \sigma_{x}(\bar x_t(\mu),\bar u_t) ) Y^{\varepsilon}_t(\lambda,\mu) + \Psi^{\varepsilon}_t(\lambda,\mu) ] \mathrm{d}W_t,
\end{split}
\end{equation}
for $\Phi^{\varepsilon}$ and $\Psi^{\varepsilon}$ as stated in the Proposition. This concludes the proof.
\end{proof}

We can now rewrite the quadratic terms in  $y^{\varepsilon}$ in the variational inequality into linear terms in $Y^{\varepsilon}$, evaluated on the diagonal in $\Lambda^2$.

\begin{proposition}\label{delta}
It holds
\begin{equation}\label{eq1}
\begin{split}
&\mathbb{E} \left [ \int_0^T \left \langle b_{xx}( \bar x_t, \bar u_t)y^{\varepsilon}_t y^{\varepsilon}_t , p_t \right \rangle_{L^2(\Lambda)} + \left \langle \sigma_{xx}( \bar x_t, \bar u_t)y^{\varepsilon}_t y^{\varepsilon}_t , q_t \right \rangle_{L_2(\Xi,L^2(\Lambda))} \mathrm{d}t \right ]\\
&+ \mathbb{E} \left [ \int_0^T \int_{\Lambda} l_{xx}(\bar x_t, \bar u_t)y^{\varepsilon}_ty^{\varepsilon}_t \mathrm{d}\lambda \mathrm{d}t \right ]\\
=& \mathbb{E} \left [ \int_0^T \int_{\Lambda} \left ( b_{xx}(\bar x_t(\lambda), \bar u_t) p_t(\lambda) + \langle \sigma_{xx}(\bar x_t(\lambda), \bar u_t), q_t(\lambda) \rangle_{L_2(\Xi,\mathbb{R})} \right ) \delta(Y^{\varepsilon}_t)(\lambda) \mathrm{d}\lambda \mathrm{d}t \right ] \\
& +\mathbb{E} \left [ \int_0^T \int_{\Lambda} l_{xx}(\bar x_t(\lambda), \bar u_t) \delta(Y^{\varepsilon}_t)(\lambda) \mathrm{d}\lambda \mathrm{d}t \right ],
\end{split}
\end{equation}
where $\delta:H^1_0(\Lambda^2) \to L^2(\Lambda)$ is defined by $\delta(w)(\lambda) := w(\lambda,\lambda)$.
\end{proposition}
\begin{proof}
Let $(\xi_k)_{k\in \mathbb{N}}$ be an orthonormal basis of $\Xi$. We have
\begin{equation}
\begin{split}
\left \langle \sigma_{xx}( \bar x_t, \bar u_t)y^{\varepsilon}_t y^{\varepsilon}_t , q_t \right \rangle_{L_2(\Xi,L^2(\Lambda))} &= \sum_{k=1}^{\infty}  \left \langle \sigma_{xx}(\bar x_t,\bar u_t)(\xi_k) y^{\varepsilon}_t y^{\varepsilon}_t , q_t(\xi_k) \right \rangle_{L^2(\Lambda)}\\
&= \sum_{k=1}^{\infty} \int_{\Lambda} \sigma_{xx}(\bar x_t(\lambda),\bar u_t)(\xi_k) y^{\varepsilon}_t(\lambda) y^{\varepsilon}_t(\lambda) q_t(\xi_k)(\lambda) \mathrm{d} \lambda\\
&= \int_{\Lambda} \delta(Y^{\varepsilon}_t)(\lambda) \langle \sigma_{xx}(\bar x_t(\lambda),\bar u_t), q_t(\lambda) \rangle_{L_2(\Xi,\mathbb{R})} \mathrm{d}\lambda.
\end{split}
\end{equation}
A similar calculation shows the claim for the remaining terms.
\end{proof}
The operator $\delta: H^1_0(\Lambda^2) \to L^2(\Lambda)$, $\Lambda\subset \mathbb{R}$, is continuous due to the trace theorem (see \cite{adams2003}, Section 7.38). Since
\begin{equation}
Y^{\varepsilon} \in L^2([0,T]\times\Omega; H^1_0(\Lambda^2)),
\end{equation}
the right-hand side of \cref{eq1} is linear and bounded in $Y^{\varepsilon}$. However, the spatial regularity of the solution evaluated at the terminal time $T$ is not sufficient for
\begin{equation}
\mathbb{E} \left [ \int_{\Lambda} h_{xx}(\bar x_T(\lambda)) Y^{\varepsilon}_T(\lambda,\lambda) \mathrm{d}\lambda \right ]
\end{equation}
to be continuous in $Y^{\varepsilon}_T$. In order to obtain a continuous operator in $Y^{\varepsilon}_T$, we have to mollify the terminal condition. Using the heat kernel, we have
\begin{equation}
\begin{split}
&\mathbb{E} \left [ \int_{\Lambda} h_{xx}(\bar x_T(\lambda)) y^{\varepsilon}_T(\lambda) y^{\varepsilon}_T(\lambda) \mathrm{d}\lambda \right ]\\
=& \lim_{\eta \to 0} \mathbb{E}\! \left [ \int_{\Lambda^2} \!\frac 12 \left( h_{xx}(\bar x_T(\lambda)) + h_{xx}(\bar x_T(\mu))\right) y^{\varepsilon}_T(\lambda)  y^{\varepsilon}_T(\mu) \frac{1}{\sqrt{4\pi \eta}} \exp \left ( - \frac{|\lambda-\mu|^2}{4\eta} \right )\mathrm{d}\mu \mathrm{d}\lambda \right ]\\
=& \lim_{\eta \to 0} \mathbb{E}\! \left [
\int_{\Lambda^2} \!
\frac 12 \left( h_{xx}(\bar x_T(\lambda)) + h_{xx}(\bar x_T(\mu))\right)
\frac{1}{\sqrt{4\pi \eta}} \exp \left ( - \frac{|\lambda-\mu|^2}{4\eta} \right ) Y^{\varepsilon}_T(\lambda,\mu) \mathrm{d}\lambda \mathrm{d}\mu \right ].
\end{split}
\end{equation}
We denote
\begin{equation}\label{meta}
h^{\eta}_{xx}(\lambda,\mu) := 
\frac 12 \left( h_{xx}(\bar x_T(\lambda)) + h_{xx}(\bar x_T(\mu))\right)
\frac{1}{\sqrt{4\pi \eta}} \exp \left ( - \frac{|\lambda-\mu|^2}{4\eta} \right ) \in L^2(\Lambda^2).
\end{equation}
With this mollification and \cref{delta}, we construct another bounded, linear functional via
\begin{equation}
\begin{split}
&\mathcal T^{\eta}_2(\Phi,\Psi) \\
:=& \mathbb{E} \left [ \int_0^T \int_{\Lambda} \left ( b_{xx}(\bar x_t(\lambda), \bar u_t) p_t(\lambda) + \langle \sigma_{xx}(\bar x_t(\lambda), \bar u_t), q_t(\lambda) \rangle_{L_2(\Xi,\mathbb{R})} \right ) \delta(Y_t)(\lambda) \mathrm{d}\lambda \mathrm{d}t \right ] \\
& +\mathbb{E} \left [ \int_0^T \int_{\Lambda} l_{xx}(\bar x_t(\lambda), \bar u_t) \delta(Y_t)(\lambda) \mathrm{d}\lambda \mathrm{d}t + \int_{\Lambda^2} h^{\eta}_{xx}(\lambda,\mu) Y_T(\lambda,\mu) \mathrm{d}\lambda \mathrm{d}\mu \right ],
\end{split}
\end{equation}
where $Y$ denotes the solution to equation \cref{secondvariationalequation} with $(\Phi^{\varepsilon},\Psi^{\varepsilon})$ replaced by $(\Phi,\Psi)$. By Riesz' representation theorem, there exists a pair 
\begin{equation}
(P^{\eta},Q^{\eta}) \in L^2([0,T]\times\Omega; L^2(\Lambda^2)) \times L^2([0,T]\times\Omega; L_2(\Xi,L^2(\Lambda^2)))
\end{equation}
such that
\begin{equation}\label{secondadjointstateproperty}
\mathcal T^{\eta}_2(\Phi, \Psi) = \mathbb{E} \left [ \int_0^T \langle P^{\eta}_t, \Phi_t \rangle_{L^2(\Lambda^2)} + \langle Q^{\eta}_t, \Psi_t \rangle_{L_2(\Xi,L^2(\Lambda^2))} \mathrm{d}t \right ],
\end{equation}
for all $(\Phi,\Psi) \in L^2([0,T]\times\Omega; L^2(\Lambda^2)) \times L^2([0,T]\times\Omega; L_2(\Xi,L^2(\Lambda^2)))$. This property is called the mollified second order adjoint state property. Choosing $\Phi = \Phi^{\varepsilon}$ and $\Psi = \Psi^{\varepsilon}$ as given by equations \cref{phi} and \cref{psi}, respectively, we can rewrite \cref{step1p} as
\begin{equation}\label{variational1}
\begin{split}
&\mathbb{E} \left [ \int_0^T \left \langle b(\bar x_t, u^{\varepsilon}_t) - b(\bar x_t, \bar u_t) , p_t \right \rangle_{L^2(\Lambda)} + \left \langle \sigma(\bar x_t, u^{\varepsilon}_t) - \sigma( \bar x_t, \bar u_t) , q_t \right \rangle_{L_2(\Xi,L^2(\Lambda))} \mathrm{d}t \right ]\\
&+ \frac12 \mathbb{E} \left [ \int_0^T \langle P^{\eta}_t, \Phi^{\varepsilon}_t \rangle_{L^2(\Lambda^2)} + \langle Q^{\eta}_t, \Psi^{\varepsilon}_t \rangle_{L_2(\Xi,L^2(\Lambda^2))} \mathrm{d}t \right ]\\
&+ \mathbb{E} \left [ \int_0^T \int_{\Lambda} l(\bar x_t(\lambda), u^{\varepsilon}_t) - l( \bar x_t(\lambda), \bar u_t) \mathrm{d}\lambda \mathrm{d}t \right ] \\
&+\frac12 \mathbb{E} \left [ \int_{\Lambda} h_{xx}( \bar x_T(\lambda)) y^{\varepsilon}_T(\lambda) y^{\varepsilon}_T(\lambda) \mathrm{d}\lambda - \int_{\Lambda^2} h^{\eta}_{xx}(\lambda,\mu) Y^{\varepsilon}_T(\lambda,\mu) \mathrm{d}\lambda \mathrm{d}\mu \right ] \geq o(\varepsilon).
\end{split}
\end{equation}

\begin{remark}\label{changesdimension}
The restriction to one space-dimension goes back to the required continuity of the operator $\delta$ defined in \cref{delta}. For two-dimensional $\Lambda \subset \mathbb{R}^2$, $\delta$ maps from $H^1_0(\Lambda^2)$ to $L^2(\Lambda)$, which means that we lose two space-dimensions and therefore lose the continuity of $\delta$. However, continuity can be restored if we have the space regularity $H^{1+\epsilon}_0(\Lambda^2)$, $\epsilon >0$, see \cite{adams2003}, Section 7.38. This can be achieved by assuming higher space-regularity on the noise coefficient $\sigma$.
\end{remark}

\section{Adjoint Equations}\label{adjointequations}
\noindent
In this section, we are going to deduce equations for the adjoint states $(p,q)$ and $(P^{\eta},Q^{\eta})$, respectively.
\subsection{First Order Adjoint Equation}

We introduce the following first order adjoint equation
\begin{equation}
\begin{cases}
\mathrm{d}p_t = -\left [ \Delta p_t + b_{x}(\bar x_t,\bar u_t) p_t + \langle \sigma_{x}(\bar x_t,\bar u_t), q_t \rangle_{L_2(\Xi,\mathbb{R})} + l_{x}(\bar x_t,\bar u_t) \right ]\mathrm{d}t + q_t \mathrm{d}W_t\\
p_T = h_{x}(\bar x_T).
\end{cases}
\end{equation}
The existence of a unique variational solution $(p,q)$ to this equation, where
\begin{equation}
p \in L^2 ([0,T]\times\Omega;H^1_0(\Lambda))\cap L^2(\Omega;C([0,T];L^2(\Lambda)))
\end{equation}
and
\begin{equation}
q\in L^2([0,T]\times\Omega;L_2(\Xi,L^2(\Lambda))),
\end{equation}
can be found in \cite{bensoussan1983}. In order to verify the adjoint state property \cref{firstadjointstateproperty} we need to apply It\^o's formula to the process $\langle p_t,y_t \rangle_{L^2(\Lambda)}$, where $y$ denotes the solution to equation \cref{riesz1p} associated with $(\varphi,\psi)$. To this end we need It\^o's formula for variational solutions to SPDEs (see page 65 in \cite{pardoux1975} or Section 3 of \cite{krylov2013}) with $V := H^1_0(\Lambda)\times H^1_0(\Lambda)$, $H := L^2(\Lambda)\times L^2(\Lambda)$, and $F:H \to \mathbb R$, $(x,y)\mapsto \langle x,y \rangle_{L^2(\Lambda)}$. This yields
\begin{equation}
\begin{split}
\mathrm{d}\langle p_t, y_t \rangle_{L^2(\Lambda)} =& \langle p_t,\mathrm{d}y_t \rangle_{L^2(\Lambda)} + \langle y_t,\mathrm{d}p_t \rangle_{L^2(\Lambda)} + \mathrm{d}\langle p,y\rangle_t\\
=& \langle p_t, b_{x}(\bar x_t, \bar u_t) y_t + \varphi_t\rangle_{L^2(\Lambda)} \mathrm{d}t\\
&- \langle y_t, b_{x}(\bar x_t, \bar u_t ) p_t + \langle \sigma_{x}(\bar x_t , \bar u_t ), q_t\rangle_{L_2(\Xi,\mathbb{R})} + l_{x}(\bar x_t, \bar u_t) \rangle_{L^2(\Lambda)} \mathrm{d}t\\
&+\langle q_t , \sigma_{x}(\bar x_t, \bar u_t ) y_t + \psi_t \rangle_{L_2(\Xi,L^2(\Lambda))} \mathrm{d}t\\
&+ \langle \left ( \sigma_{x}(\bar x_t, \bar u_t ) y_t + \psi_t \right )^{\ast} p_t , \mathrm{d}W_t \rangle_{L^2(\Lambda)} + \langle q_t^{\ast} y_t, \mathrm{d}W_t \rangle_{L^2(\Lambda)}\\
=& \left [ \langle p_t,\varphi_t \rangle_{L^2(\Lambda)} + \langle q_t,\psi_t \rangle_{L_2(\Xi,L^2(\Lambda))} - \langle y_t,l_{x}(\bar x_t,\bar u_t) \rangle_{L^2(\Lambda)} \right ] \mathrm{d}t\\
&+ \langle \left ( \sigma_{x}(\bar x_t, \bar u_t ) y_t + \psi_t \right )^{\ast} p_t , \mathrm{d}W_t \rangle_{L^2(\Lambda)} + \langle q_t^{\ast} y_t, \mathrm{d}W_t \rangle_{L^2(\Lambda)}. 
\end{split}
\end{equation}
Hence, considering the terminal condition, we obtain
\begin{equation}
\begin{split}
&\mathbb{E} \left [ \langle h_{x}(\bar x_T),y_T \rangle_{L^2(\Lambda)} \right ]\\
=& \mathbb{E} \left [ \int_0^T \langle p_t,\varphi_t \rangle_{L^2(\Lambda)} + \langle q_t,\psi_t \rangle_{L_2(\Xi,L^2(\Lambda))} - \langle y_t,l_{x}(\bar x_t,\bar u_t) \rangle_{L^2(\Lambda)} \mathrm{d}t \right ],
\end{split}
\end{equation}
which is the adjoint state property.

The strategy for the second order adjoint state is the same: First, we introduce the second order adjoint equation and show that a solution to that equation exists; then we apply It\^o's formula and show that the solution satisfies the mollified adjoint state property \cref{secondadjointstateproperty}, which characterizes it as the mollified second order adjoint state.

\subsection{Mollified Second Order Adjoint Equation}

We introduce the mollified second order adjoint equation
\begin{equation}\label{secondorderadjoint}
\begin{cases}
\mathrm{d}P^{\eta}_t(\lambda,\mu) = - [ \Delta P^{\eta}_t(\lambda,\mu) + ( b_{x}(\bar x_t(\lambda),\bar u_t) + b_{x}(\bar x_t(\mu),\bar u_t) ) P^{\eta}_t(\lambda,\mu)\\
\qquad\qquad\qquad + \langle \sigma_{x}(\bar x_t(\lambda), \bar u_t), \sigma_{x}(\bar x_t(\mu), \bar u_t) \rangle_{L_2(\Xi,\mathbb{R})} P^{\eta}_t(\lambda,\mu) \\
\qquad\qquad\qquad + \langle \sigma_{x} (\bar x_t(\lambda),\bar u_t) + \sigma_{x} (\bar x_t(\mu),\bar u_t), Q^{\eta}_t(\lambda,\mu)\rangle_{L_2(\Xi,\mathbb{R})}\\
\qquad\qquad\qquad + \delta^{\ast}( l_{xx}(\bar x_t(\lambda), \bar u_t) ) + \delta^{\ast}( b_{xx}(\bar x_t(\lambda), \bar u_t) p_t(\lambda) )\\
\qquad\qquad\qquad + \delta^{\ast}( \langle \sigma_{xx}(\bar x_t(\lambda),\bar u_t), q_t \rangle_{L_2(\Xi,\mathbb{R})} ) ] \mathrm{d}t + Q^{\eta}_t(\lambda,\mu) \mathrm{d}W_t\\
P^{\eta}_T(\lambda,\mu) = h^{\eta}_{xx}(\lambda,\mu),
\end{cases}
\end{equation}
where $h^{\eta}_{xx}$ is given by equation \cref{meta}, and $\delta^{\ast} : L^2(\Lambda) \to H^{-1}(\Lambda^2)$ is the adjoint of the operator introduced in \cref{delta}, i.e.
\begin{equation}
\langle \delta^{\ast}(f) , w \rangle_{H^{-1}(\Lambda^2)\times H^1_0(\Lambda^2)} := \int_{\Lambda} f(\lambda) \delta(w)(\lambda) \mathrm{d}\lambda = \int_{\Lambda} f(\lambda) w(\lambda,\lambda) \mathrm{d}\lambda,
\end{equation}
for $f\in L^2(\Lambda)$, $w\in H^1_0(\Lambda^2)$.

\begin{proposition}
The mollified second order adjoint equation \cref{secondorderadjoint} has a unique variational solution $(P^{\eta}, Q^{\eta})$, where
\begin{equation}
P^{\eta} \in L^2([0,T]\times\Omega; H^1_0(\Lambda^2)) \cap L^2(\Omega; C([0,T];L^2(\Lambda)))
\end{equation}
and
\begin{equation}
Q^{\eta} \in L^2([0,T]\times\Omega;L_2(\Xi,L^2(\Lambda^2))).
\end{equation}
\end{proposition}

\begin{proof}
We apply the result from \cite{bensoussan1983} on the Gelfand triple
\begin{equation}
H^1_0(\Lambda^2) \hookrightarrow L^2(\Lambda^2) \hookrightarrow H^{-1}(\Lambda^2).
\end{equation}
\end{proof}

\subsection{Adjoint State Property for the Mollified Second Order Adjoint State}

Now, we are going to show that the solution to the mollified second order adjoint equation satisfies the mollified adjoint state property \cref{secondadjointstateproperty}. To this end let $Y$ denote the solution to the second variational equation \cref{secondvariationalequation} associated with $(\Phi,\Psi)$, and let $(P^{\eta},Q^{\eta})$ denote the solution to the mollified second order adjoint equation \cref{secondorderadjoint}. We again apply It\^o's formula for variational solutions to SPDEs, this time to the expression
\begin{equation}
\langle P^{\eta}_t(\lambda,\mu),Y_t(\lambda,\mu) \rangle_{L^2(\Lambda^2)}. 
\end{equation}
Choosing $V := H^1_0(\Lambda^2)\times H^1_0(\Lambda^2)$, $H := L^2(\Lambda^2)\times L^2(\Lambda^2)$, and $F:H \to \mathbb R$, $(x,y)\mapsto \langle x,y \rangle_{L^2(\Lambda^2)}$, yields
\begin{equation}
\begin{split}
&\mathrm{d}\langle P^{\eta}_t(\lambda,\mu),Y_t(\lambda,\mu) \rangle_{L^2(\Lambda^2)}\\
=& \langle P^{\eta}_t(\lambda,\mu),\mathrm{d}Y_t(\lambda,\mu) \rangle_{L^2(\Lambda^2)} + \langle Y_t(\lambda,\mu),\mathrm{d}P^{\eta}_t(\lambda,\mu) \rangle_{L^2(\Lambda^2)} + \mathrm{d} \langle P^{\eta}(\lambda,\mu),Y(\lambda,\mu)\rangle_t.
\end{split}
\end{equation}
Plugging in the equations for $P^{\eta}$ and $Y$, respectively, we arrive at
\begin{equation}
\begin{split}
& \mathrm{d}\langle P^{\eta}_t(\lambda,\mu),Y_t(\lambda,\mu) \rangle_{L^2(\Lambda^2)}\\
=&\Big \langle (\Delta Y_t(\lambda,\mu) + ( b_{x}(\bar x_t(\lambda),\bar u_t) + b_{x}(\bar x_t(\mu),\bar u_t) ) Y_t(\lambda,\mu)\\
&\qquad+ \langle \sigma_{x}(\bar x_t(\lambda),\bar u_t), \sigma_{x}(\bar x_t(\mu),\bar u_t)\rangle_{L_2(\Xi,\mathbb{R})} Y_t(\lambda,\mu)\\
&\qquad+ \Phi_t(\lambda,\mu)) , P^{\eta}_t(\lambda,\mu) \Big \rangle_{H^{-1}(\Lambda^2)\times H^1_0(\Lambda^2)} \mathrm{d}t\\
&- \Big \langle \Delta P^{\eta}_t(\lambda,\mu) + ( b_{x}(\bar x_t(\lambda),\bar u_t) + b_{x}(\bar x_t(\mu),\bar u_t) ) P^{\eta}_t(\lambda,\mu)\\
&\qquad+ \langle \sigma_{x}(\bar x_t(\lambda), \bar u_t), \sigma_{x}(\bar x_t(\mu), \bar u_t) \rangle_{L_2(\Xi,\mathbb{R})} P^{\eta}_t(\lambda,\mu) \\
&\qquad+ \langle \sigma_{x} (\bar x_t(\lambda),\bar u_t) + \sigma_{x} (\bar x_t(\mu),\bar u_t), Q^{\eta}_t(\lambda,\mu)\rangle_{L_2(\Xi,\mathbb{R})}\\
&\qquad+ \delta^{\ast}( l_{xx}(\bar x_t, \bar u_t) )  + \delta^{\ast}( b_{xx}(\bar x_t, \bar u_t) p_t(\lambda) )\\
&\qquad+ \delta^{\ast}( \langle \sigma_{xx}(\bar x_t, \bar u_t), q_t(\lambda) \rangle_{L_2(\Xi,\mathbb{R})}), Y_t(\lambda,\mu) \Big \rangle_{H^{-1}(\Lambda^2)\times H^1_0(\Lambda^2)} \mathrm{d}t\\
&+ \langle Q^{\eta}_t(\lambda,\mu), ( \sigma_{x}(\bar x_t(\lambda),\bar u_t) + \sigma_{x}(\bar x_t(\mu),\bar u_t) ) Y_t(\lambda,\mu) + \Psi_t(\lambda,\mu) \rangle_{L_2(\Xi,L^2(\Lambda^2))} \mathrm{d}t\\
&+ \langle ( ( \sigma_{x}(\bar x_t(\lambda),\bar u_t) + \sigma_{x}(\bar x_t(\mu),\bar u_t) ) Y_t(\lambda,\mu) + \Psi_t(\lambda,\mu) )^{\ast} Q^{\eta}_t(\lambda,\mu), \mathrm{d}W_t \rangle_{L^2(\Lambda^2)}\\
&+ \langle Q^{\eta}_t(\lambda,\mu)^{\ast} ( ( \sigma_{x}(\bar x_t(\lambda),\bar u_t) + \sigma_{x}(\bar x_t(\mu),\bar u_t) ) Y_t(\lambda,\mu) + \Psi_t(\lambda,\mu) ), \mathrm{d}W_t \rangle_{L^2(\Lambda^2)}.
\end{split}
\end{equation}
Integrating the Laplacian by parts yields
\begin{equation}
\begin{split}
& \mathrm{d}\langle P^{\eta}_t(\lambda,\mu),Y_t(\lambda,\mu) \rangle_{L^2(\Lambda^2)}\\
=& \Big [ \langle P^{\eta}_t(\lambda,\mu), \Phi_t(\lambda,\mu) \rangle_{L^2(\Lambda^2)} + \langle Q^{\eta}_t(\lambda,\mu), \Psi_t(\lambda,\mu) \rangle_{L_2(\Xi,L^2(\Lambda^2))}\\
&- \langle \delta^{\ast}( l_{xx}(\bar x_t, \bar u_t) )  + \delta^{\ast}( b_{xx}(\bar x_t, \bar u_t) p_t(\lambda) )\\
&\qquad + \delta^{\ast}( \langle \sigma_{xx}(\bar x_t, \bar u_t), q_t(\lambda) \rangle_{L_2(\Xi,\mathbb{R})}), Y_t(\lambda,\mu)  \rangle_{H^{-1}(\Lambda^2)\times H^1_0(\Lambda^2)} \Big ] \mathrm{d}t\\
&+ \langle ( ( \sigma_{x}(\bar x_t(\lambda),\bar u_t) + \sigma_{x}(\bar x_t(\mu),\bar u_t) ) Y_t(\lambda,\mu) + \Psi_t(\lambda,\mu) )^{\ast} Q^{\eta}_t(\lambda,\mu), \mathrm{d}W_t \rangle_{L^2(\Lambda^2)}\\
&+ \langle Q^{\eta}_t(\lambda,\mu)^{\ast} ( ( \sigma_{x}(\bar x_t(\lambda),\bar u_t) + \sigma_{x}(\bar x_t(\mu),\bar u_t) ) Y_t(\lambda,\mu) + \Psi_t(\lambda,\mu) ), \mathrm{d}W_t \rangle_{L^2(\Lambda^2)}.
\end{split}
\end{equation}
Therefore, taking expectations and considering the initial and terminal condition for $Y$ and $P^{\eta}$, respectively, we obtain
\begin{equation}\label{variational}
\begin{split}
&\mathbb{E} \left [ \langle h^{\eta}_{xx}, Y_T \rangle_{L^2(\Lambda^2)} \right ]\\
=& \mathbb{E} \Big [ \int_0^T \langle P^{\eta}_t,\Phi_t \rangle_{L^2(\Lambda^2)} + \langle Q^{\eta}_t, \Psi_t \rangle_{L^2(\Xi,L^2(\Lambda^2))}\\
&\qquad - \langle \delta^{\ast}( l_{xx}(\bar x_t, \bar u_t) )  + \delta^{\ast}( b_{xx}(\bar x_t, \bar u_t) p_t(\lambda) )\\
&\qquad+ \delta^{\ast}( \langle \sigma_{xx}(\bar x_t, \bar u_t), q_t(\lambda) \rangle_{L_2(\Xi,\mathbb{R})}, Y_t(\lambda,\mu)  \rangle_{H^{-1}(\Lambda^2)\times H^1_0(\Lambda^2)} \mathrm{d}t \Big ],
\end{split}
\end{equation}
which is the mollified second order adjoint state property \cref{secondadjointstateproperty}. Hence, the mollified second order adjoint state is characterized by equation \cref{secondorderadjoint}.

\section{Passing to the Limit of the Mollified Second Order Adjoint State}\label{limit}
In this section, we derive an equation for the second order adjoint state $P= \lim_{\eta\to 0} P^{\eta}$. Recall that we chose $h^{\eta}_{xx}$ in such a way, that
\begin{equation}
\lim_{\eta \to 0} h^{\eta}_{xx} = \delta^{\ast} (h_{xx}(\bar x_T(\lambda))) \quad \text{in}\; H^{-1}(\Lambda^2).
\end{equation}

\begin{theorem}\label{existencegeneralsecondorderadjointequation}
The equation
\begin{equation}\label{generalsecondorderadjoint}
\begin{cases}
\mathrm{d}P_t(\lambda,\mu) = - [ \Delta P_t(\lambda,\mu) + ( b_{x}(\bar x_t(\lambda),\bar u_t) + b_{x}(\bar x_t(\mu),\bar u_t) ) P_t(\lambda,\mu)\\
\qquad\qquad\qquad + \langle \sigma_{x}(\bar x_t(\lambda), \bar u_t), \sigma_{x}(\bar x_t(\mu), \bar u_t) \rangle_{L_2(\Xi,\mathbb{R})} P_t(\lambda,\mu) \\
\qquad\qquad\qquad + \langle \sigma_{x} (\bar x_t(\lambda),\bar u_t) + \sigma_{x} (\bar x_t(\mu),\bar u_t), Q_t(\lambda,\mu)\rangle_{L_2(\Xi,\mathbb{R})}\\
\qquad\qquad\qquad + \delta^{\ast}( l_{xx}(\bar x_t(\lambda), \bar u_t) ) + \delta^{\ast}( b_{xx}(\bar x_t(\lambda), \bar u_t) p_t(\lambda) )\\
\qquad\qquad\qquad + \delta^{\ast}( \langle \sigma_{xx}(\bar x_t(\lambda),\bar u_t), q_t \rangle_{L_2(\Xi,\mathbb{R})} ) ] \mathrm{d}t + Q_t(\lambda,\mu) \mathrm{d}W_t\\
P_T(\lambda,\mu) = \delta^{\ast}( h_{xx}(\bar x_T(\lambda)) )
\end{cases}
\end{equation}
has a unique adapted solution $(P,Q)$, where
\begin{equation}
P \in L^2([0,T]\times\Omega;L^2(\Lambda^2)) \cap L^2(\Omega; C([0,T];H^{-1}(\Lambda^2))),
\end{equation}
and
\begin{equation}
Q\in L^2([0,T]\times\Omega; L_2(\Xi;H^{-1}(\Lambda^2))).
\end{equation}
Here equation \cref{generalsecondorderadjoint} holds in $L^2([0,T]\times\Omega;H^{-2}(\Lambda^2))$.
\end{theorem}

\begin{proof}
First, we prove existence of a solution. Let $(P^{\eta},Q^{\eta})$ denote the solution to equation \cref{secondorderadjoint}. We define $F:H^{-1}(\Lambda^2) \to \mathbb{R}$, $x\mapsto \|x\|_{H^{-1}(\Lambda^2)}^2$. Since $P^{\eta}$ is an $H^{-1}(\Lambda^2)$-valued semimartingale, we can apply the classical version of It\^o's formula for Hilbert space-valued semimartingales (see \cite{daprato1992}, Section 4.4), which yields
\begin{equation}\label{classicito}
\begin{split}
&\| P^{\eta}_t(\lambda,\mu) \|_{H^{-1}(\Lambda^2)}^2\\
= &\| h^{\eta}_{xx} \|_{H^{-1}(\Lambda^2)}^2 + 2\int_t^T \langle \Delta P^{\eta}_s(\lambda,\mu), P^{\eta}_s(\lambda,\mu) \rangle_{H^{-1}(\Lambda^2)} \mathrm{d}s\\
&+ 2\int_t^T \langle ( b_{x}(\bar x_s(\lambda),\bar u_s) + b_{x}(\bar x_s(\mu),\bar u_s)) P^{\eta}_s(\lambda,\mu) , P^{\eta}_s(\lambda,\mu) \rangle_{H^{-1}(\Lambda^2)} \mathrm{d}s\\
&+ 2\int_t^T \langle \langle \sigma_{x}(\bar x_s(\lambda), \bar u_s), \sigma_{x}(\bar x_s(\mu), \bar u_s) \rangle_{L_2(\Xi,\mathbb{R})} P^{\eta}_s(\lambda,\mu) , P^{\eta}_s(\lambda,\mu) \rangle_{H^{-1}(\Lambda^2)} \mathrm{d}s\\
&+ 2\int_t^T \langle \langle \sigma_{x} (\bar x_s(\lambda),\bar u_s) + \sigma_{x} (\bar x_s(\mu),\bar u_s), Q^{\eta}_s(\lambda,\mu)\rangle_{L_2(\Xi,\mathbb{R})} , P^{\eta}_s(\lambda,\mu) \rangle_{H^{-1}(\Lambda^2)} \mathrm{d}s\\
&+ 2\int_t^T \langle \delta^{\ast}( l_{xx}(\bar x_s(\lambda), \bar u_s) ) + \delta^{\ast}( b_{xx}(\bar x_s(\lambda), \bar u_s) p_s(\lambda) )\\
&\qquad\qquad\qquad + \delta^{\ast}( \langle \sigma_{xx}(\bar x_s(\lambda),\bar u_s), q_s \rangle_{L_2(\Xi,\mathbb{R})} ) , P^{\eta}_s(\lambda,\mu) \rangle_{H^{-1}(\Lambda^2)} \mathrm{d}s\\
&- \int_t^T \|Q^{\eta}_s \|^2_{L_2(\Xi,H^{-1}(\Lambda^2))} \mathrm{d}s + 2\int_t^T \langle P^{\eta}_s(\lambda,\mu), Q^{\eta}_s(\lambda,\mu) \mathrm{d}W_s \rangle_{H^{-1}(\Lambda^2)}.
\end{split}
\end{equation}
By Lemma 4.1.12 in \cite{liu2015}, we have
\begin{equation}
\langle \Delta P^{\eta}_s(\lambda,\mu), P^{\eta}_s(\lambda,\mu) \rangle_{H^{-1}(\Lambda^2)} = - \| P^{\eta}_s(\lambda,\mu) \|_{L^2(\Lambda^2)}^ 2.
\end{equation}
Therefore, from equation \cref{classicito} we derive
\begin{equation}
\begin{split}
&\| P^{\eta}_t(\lambda,\mu) \|_{H^{-1}(\Lambda^2)}^2 + 2\int_t^T \| P^{\eta}_s(\lambda,\mu) \|_{L^2(\Lambda^2)}^2 \mathrm{d}s + \int_t^T \|Q^{\eta}_s \|^2_{L_2(\Xi,H^{-1}(\Lambda^2))} \mathrm{d}s\\
\leq & \| h^{\eta}_{xx} \|_{H^{-1}(\Lambda^2)}^2 + C(b,\sigma,T,l) \left (1+ \int_t^T \| P^{\eta}_s(\lambda,\mu) \|_{H^{-1}(\Lambda^2)}^2 \mathrm{d}s \right )\\
& + 2\int_t^T \langle P^{\eta}_s(\lambda,\mu), Q^{\eta}_s(\lambda,\mu) \mathrm{d}W_s \rangle_{H^{-1}(\Lambda^2)}.
\end{split}
\end{equation}
Taking the supremum and expectations, using Burkholder-Davis-Gundy inequality for the stochastic integral, and applying Gr\"onwall's inequality, we obtain
\begin{equation}\label{apriori}
\begin{split}
&\mathbb{E} \left [ \sup_{t\in[0,T]} \| P_t^{\eta}(\lambda,\mu) \|_{H^{-1}(\Lambda^2)}^2 + 2\!\int_0^T \| P^{\eta}_s(\lambda,\mu) \|_{L^2(\Lambda^2)}^2 \mathrm{d}s + \!\int_0^T \| Q^{\eta}_s \|_{L_2(\Xi,H^{-1}(\Lambda^2))}^2 \mathrm{d}s \right ] \\
\leq &C \left ( 1+ \mathbb{E} \left [ \| h^{\eta}_{xx} \|_{H^{-1}(\Lambda^2)}^2 \right ] \right ),
\end{split}
\end{equation}
where the right-hand side is uniformly bounded in $\eta$. Therefore, we can extract weakly convergent subsequences
\begin{align}
P^{\eta} &\rightharpoonup P \quad \text{in}\; L^2([0,T]\times\Omega; L^2(\Lambda^2)),\\
Q^{\eta} &\rightharpoonup Q \quad\text{in}\; L^2([0,T]\times\Omega; L_2(\Xi,H^{-1}(\Lambda^2))),
\end{align}
which implies
\begin{equation}
\int_{\cdot}^T Q^{\eta}_s \mathrm{d}W_s \stackrel{*}{\rightharpoonup} \int_{\cdot}^T Q_s \mathrm{d}W_s \quad\text{in}\; L^{\infty}([0,T]; L^2(\Omega;H^{-1}(\Lambda^2))).
\end{equation}
Since $\Delta : L^2(\Lambda^2) \to H^{-2}(\Lambda^2)$ is weak-weak continuous, we can test the mollified second order adjoint equation \cref{secondorderadjoint} with a test function in $H^2_0(\Lambda^2)$ and pass to the limit $\eta \to 0$, which concludes the proof of existence. The continuity of $P$ as a process with values in $H^{-1}(\Lambda^2)$ follows from Theorem 4.2.5 in \cite{liu2015}. In order to prove uniqueness, we observe that, by the linearity of the equation, the difference of two solutions satisfies the corresponding equation with vanishing inhomogeneity and terminal condition. Hence, by an analogous argument as the one for the a priori bound \cref{apriori}, the two solutions must coincide. 
\end{proof}
\begin{remark}\label{changesoperator2}
In case the state equation \cref{stateequation} is governed by the more general uniformly elliptic differential operator $A$ given in equation \cref{generaloperator}, the Laplacian in equation \cref{generalsecondorderadjoint} is replaced by the operator $\bar A : H^1_0(\Lambda^2) \to H^{-1}(\Lambda^2)$, 
\begin{equation}
\bar Ax (\lambda,\mu ) := ( \partial_{\lambda} ( a(\lambda) \partial_{\lambda} x) + \partial_{\mu} ( a(\mu) \partial_{\mu} x))(\lambda,\mu).
\end{equation}
Therefore, we have to consider the functional
\begin{align}
F: H^{-1}(\Lambda^2)& \to \mathbb{R} \\
x&\mapsto \|x\|_{\mathcal{D}((-\bar A )^{-\frac12})}^2 = \| (I-\bar A)^{-\frac12} x \|^2_{L^2(\Lambda^2)}.
\end{align}
In this case, the term
\begin{equation}
\langle \Delta P^{\eta}_s(\lambda,\mu), P^{\eta}_s(\lambda,\mu) \rangle_{H^{-1}(\Lambda^2)}
\end{equation}
is replaced by
\begin{equation}
\langle \bar A P^{\eta}_s(\lambda,\mu), P^{\eta}_s(\lambda,\mu) \rangle_{\mathcal{D}((-\bar A)^{-\frac12})} = - \| P^{\eta}_s \|_{L^2(\Lambda^2)}^2 + \| P^{\eta}_s \|_{\mathcal{D}((-\bar  A )^{-\frac12})}^2.
\end{equation}
Now, using the same arguments as in the preceding proof, we obtain the corresponding result for $\bar A$.
\end{remark}
The following property of the second order adjoint state is not used hereafter, but is of independent interest.
\begin{proposition}\label{secondorderadjointproperty}
It holds
\begin{equation}
\begin{split}
& \mathbb{E} \left [ \int_0^T \int_{\Lambda} \left ( b_{xx}(\bar x_t(\lambda), \bar u_t) p_t(\lambda) + \langle \sigma_{xx}(\bar x_t(\lambda), \bar u_t), q_t(\lambda) \rangle_{L_2(\Xi,\mathbb{R})} \right ) y^{\varepsilon}_t(\lambda) y^{\varepsilon}_t(\lambda) \mathrm{d}\lambda \mathrm{d}t \right ] \\
& +\mathbb{E} \left [ \int_0^T \int_{\Lambda} l_{xx}(\bar x_t(\lambda), \bar u_t) y^{\varepsilon}_t(\lambda) y^{\varepsilon}_t(\lambda) \mathrm{d}\lambda \mathrm{d}t + \int_{\Lambda} h_{xx}(\bar x_T(\lambda)) y^{\varepsilon}_T(\lambda) y^{\varepsilon}_T(\lambda) \mathrm{d}\lambda \right ]\\
=& \mathbb{E} \left [ \int_0^T \langle P_t, \Phi^{\varepsilon}_t \rangle_{L^2(\Lambda^2)} + \langle Q_t, \Psi^{\varepsilon}_t \rangle_{L_2(\Xi,H^{-1}(\Lambda^2))\times L_2(\Xi,H^1_0(\Lambda^2))} \mathrm{d}t \right ],
\end{split}
\end{equation}
where $y^{\varepsilon}$ is the solution to equation \cref{yepsilon} and $\Phi^{\varepsilon}$ and $\Psi^{\varepsilon}$ are given by equations \cref{phi} and \cref{psi}, respectively.
\end{proposition}
\begin{proof}
Since for $\eta \to 0$,
\begin{equation}
\mathbb{E} \left [  \int_{\Lambda^2} h^{\eta}_{xx}(\lambda,\mu) Y^{\varepsilon}_T(\lambda,\mu) \mathrm{d}\lambda \mathrm{d}\mu \right ] \to \mathbb{E} \left [ \int_{\Lambda} h_{xx}( \bar x_T(\lambda)) y^{\varepsilon}_T(\lambda) y^{\varepsilon}_T(\lambda) \mathrm{d}\lambda \right ],
\end{equation}
taking the limit $\eta \to 0$ in equation \cref{secondadjointstateproperty} yields the claim.
\end{proof}

\section{Main Result}\label{mainresult}
\noindent
In order to prove the stochastic maximum principle, we have to take the limits $\eta \to 0$ and $\varepsilon \to 0$ in the variational inequality \cref{variational1}. If we take the limit $\eta \to 0$ first, we get rid of the terms involving the terminal condition. However, the remaining term
\begin{equation}
\mathbb{E} \left [ \int_0^T \langle P_t, \Phi^{\varepsilon}_t \rangle_{L^2(\Lambda^2)} + \langle Q_t, \Psi^{\varepsilon}_t \rangle_{L_2(\Xi,H^{-1}(\Lambda^2))\times L_2(\Xi,H^1_0(\Lambda^2))} \mathrm{d}t \right ]
\end{equation}
does not have the asymptotic needed in \eqref{limitepsilon}. Indeed, the lacking regularity of $Q$ requires us to control $\Psi^{\varepsilon}$ in $L_2(\Xi,H^1_0(\Lambda^2))$. Since we can't expect such a control in general, we have to interchange the limits in $\eta$ and $\varepsilon$. In order to ensure convergence in the converse order, we need compactness of $y^{\varepsilon}_T$, $\varepsilon>0$, in $L^2(\Lambda)$.

\begin{lemma}\label{subsequence}
For $\gamma \in (0,1/2)$ and $\varepsilon \in (0, T-\tau)$, it holds
\begin{equation}
\mathbb{E} \left [ \| y^{\varepsilon}_T \|_{H_0^{\gamma}(\Lambda)}^2 \right ] \leq C \varepsilon.
\end{equation}
\end{lemma}

\begin{proof}
Set $\tilde y^{\varepsilon}_t := y^{\varepsilon}_t/\sqrt{\varepsilon}$, $t\in [0,T]$. Then $\tilde y^{\varepsilon}$ satisfies the equation
\begin{equation}
\begin{cases}
\mathrm{d}\tilde y^{\varepsilon}_t = \left [ \Delta \tilde y^{\varepsilon}_t + b_{x}(\bar x_t, \bar u_t)\tilde y^{\varepsilon}_t + \frac{1}{\varepsilon} \left ( b(\bar x_t, u^{\varepsilon}_t) - b(\bar x_t, \bar u_t) \right ) \right ] \mathrm{d}t\\
\qquad\qquad + \left [ \sigma_{x}(\bar x_t, \bar u_t) \bar y^{\varepsilon}_t + \frac{1}{\varepsilon} \left ( \sigma(\bar x_t, u^{\varepsilon}_t) - \sigma(\bar x_t, \bar u_t) \right ) \right ] \mathrm{d}W_t\\
\tilde y^{\varepsilon}_0 = 0,
\end{cases}
\end{equation}
Duhamel's formula for mild solutions yields
\begin{equation}
\begin{split}
\tilde y^{\varepsilon}_T = &\int_0^T S_{T-s} \left ( b_x(\bar x_s,\bar u_s) \tilde y^{\varepsilon}_s \right ) \mathrm{d}s + \frac{1}{\varepsilon} \int_{\tau}^{\tau+\varepsilon} S_{T-s} \left ( b( \bar x_s, v) - b( \bar x_s, \bar u_s) \right ) \mathrm{d}s\\
& \int_0^T S_{T-s} \left ( \sigma_x(\bar x_s,\bar u_s) \tilde y^{\varepsilon}_s \right ) \mathrm{d}W_s + \frac{1}{\varepsilon} \int_{\tau}^{\tau+\varepsilon} S_{T-s} \left ( \sigma( \bar x_s, v) - \sigma( \bar x_s, \bar u_s) \right ) \mathrm{d}W_s.
\end{split}
\end{equation}
Notice that the variational solution and the mild solution coincide, see \cite{hairer2009}, Proposition 5.7. By analyticity, we have the bound
\begin{equation}
\| S_t f \|_{H_0^{\gamma}(\Lambda)}^2 \leq \frac{C}{t^{2\gamma}} \| f\|_{L^2(\Lambda)}^2
\end{equation}
for any $f\in L^2(\Lambda)$, see \cite{pazy1983} Chapter 2, Lemma 6.13. Using this property and the boundedness of $b_x$, we can estimate
\begin{equation}
\mathbb{E} \left [ \left \| \int_0^T S_{T-s} \left ( b_x(\bar x_s,\bar u_s) \tilde y^{\varepsilon}_s \right ) \mathrm{d}s \right \|_{H_0^{\gamma}(\Lambda)}^2 \right ] \leq \sup_{t\in [0,T]} \mathbb{E} \left [ \| \tilde y^{\varepsilon}_t \|_{L^2(\Lambda)}^2 \right ] \int_0^T \frac{C}{(T-s)^{2\gamma}} \mathrm{d}s < \infty.
\end{equation}
Furthermore, for the second integral, we obtain
\begin{equation}
\begin{split}
&\mathbb{E} \left [ \left \| \frac{1}{\varepsilon} \int_{\tau}^{\tau+\varepsilon} S_{T-s} \left ( b( \bar x_s, v) - b( \bar x_s, \bar u_s) \right ) \mathrm{d}s \right \|_{H_0^{\gamma}(\Lambda)}^2 \right ]\\
\leq& \frac{C}{\varepsilon} \mathbb{E} \left [ \int_{\tau}^{\tau+\varepsilon} \frac{1}{(T-s)^{2\gamma}} \left \| b(\bar x_s,v)-b(\bar x_s,\bar u_s) \right \|_{L^2(\Lambda)}^2 \mathrm{d}s \right ]
\end{split}
\end{equation}
Since $\tau <T$, using the bounds on $b$ we obtain
\begin{equation}
\begin{split}
&\frac{C}{\varepsilon} \mathbb{E} \left [ \int_{\tau}^{\tau+\varepsilon} \frac{1}{(T-s)^{2\gamma}} \left \| b(\bar x_s,v)-b(\bar x_s,\bar u_s) \right \|_{L^2(\Lambda)}^2 \mathrm{d}s \right ]\\
\leq& \frac{C}{\varepsilon} \int_{\tau}^{\tau+\varepsilon} \mathbb{E} \left [ 1+\| \bar x_s \|_{L^2(\Lambda)}^2 + \| v\|_{\mathcal{U}}^2 + \| \bar u_s \|_{\mathcal{U}}^2 \right ] \mathrm{d}s\\
\leq& C\left ( 1 + \sup_{t\in [0,T]} \mathbb{E} \left [ \| \bar x_t \|_{L^2(\Lambda)}^2 \right ] + \| v\|_{\mathcal{U}}^2 + \sup_{t\in [0,T]} \mathbb{E} \left [ \| \bar u_t \|_{\mathcal{U}}^2 \right ] \right ) < \infty.
\end{split}
\end{equation}
Now we consider the first stochastic integral.
\begin{equation}
\begin{split}
&\mathbb{E} \left [ \left \| \int_0^T S_{T-s} \left ( \sigma_x(\bar x_s,\bar u_s) \tilde y^{\varepsilon}_s \right ) \mathrm{d}W_s \right \|_{H_0^{\gamma}(\Lambda)}^2 \right ]\\
=& \mathbb{E} \left [ \int_0^T \left \| S_{T-s} ( \sigma_x(\bar x_s,\bar u_s) \tilde y^{\varepsilon}_s) \right \|_{L_2(\Xi,H^{\gamma}_0(\Lambda))}^2 \mathrm{d}s \right ]\\
\leq& \mathbb{E} \left [ \int_0^T \| S_{T-s} \|_{L(L^2(\Lambda),H^{\gamma}_0(\Lambda))}^2 \| \sigma_x(\bar x_s,\bar u_s) \|_{L(L^2(\Lambda),L_2(\Xi,L^2(\Lambda)))} \| \tilde y^{\varepsilon}_s \|_{L^2(\Lambda)}^2 \mathrm{d}s \right ],
\end{split}
\end{equation}
which can be controlled by the same arguments as for the corresponding term with $b$, since $\sigma_x$ is bounded as well. Finally, for the second stochastic integral, we have
\begin{equation}
\begin{split}
&\mathbb{E} \left [ \left \| \frac{1}{\varepsilon} \int_{\tau}^{\tau+\varepsilon} S_{T-s} \left ( \sigma( \bar x_s, v) - \sigma( \bar x_s, \bar u_s) \right ) \mathrm{d}W_s \right \|_{H_0^{\gamma}(\Lambda)}^2 \right ]\\
=&\mathbb{E} \left [ \frac{1}{\varepsilon} \int_{\tau}^{\tau+\varepsilon} \left \| S_{T-s} \left ( \sigma( \bar x_s, v) - \sigma( \bar x_s, \bar u_s) \right ) \right \|_{L_2(\Xi,H_0^{\gamma}(\Lambda))}^2 \mathrm{d}s \right ]
\end{split}
\end{equation}
which is again finite by the same arguments as for the corresponding term with $b$.
\end{proof}

Now we are able to prove our main result.

\begin{theorem}[Stochastic Maximum Principle]
Let $(\bar x, \bar u)$ be an optimal pair of the control problem \cref{costfunctional} and \cref{stateequation}. Then there exist adapted processes $(p,q)$, where
\begin{equation}
p \in L^2 ([0,T]\times\Omega;H^1_0(\Lambda))\cap L^2(\Omega;C([0,T];L^2(\Lambda)))
\end{equation}
and
\begin{equation}
q\in L^2([0,T]\times\Omega;L_2(\Xi,L^2(\Lambda))),
\end{equation}
satisfying the first order adjoint equation
\begin{equation}
\begin{cases}
\mathrm{d}p_t = -\left [ \Delta p_t + b_{x}(\bar x_t,\bar u_t) p_t + \langle \sigma_{x}(\bar x_t,\bar u_t), q_t \rangle_{L_2(\Xi,\mathbb{R})} + l_{x}(\bar x_t,\bar u_t) \right ]\mathrm{d}t + q_t \mathrm{d}W_t\\
p_t = h_{x}(\bar x_t),
\end{cases}
\end{equation}
and adapted processes $(P,Q)$, where
\begin{equation}
P \in L^2([0,T]\times\Omega;L^2(\Lambda^2)) \cap L^2(\Omega;C([0,T];H^{-1}(\Lambda)))
\end{equation}
and
\begin{equation}
Q\in L^2([0,T]\times\Omega;L_2(\Xi,H^{-1}(\Lambda^2))),
\end{equation}
satisfying the second order adjoint equation
\begin{equation}
\begin{cases}
\mathrm{d}P_t(\lambda,\mu) = - [ \Delta P_t(\lambda,\mu) + ( b_{x}(\bar x_t(\lambda),\bar u_t)) + b_{x}(\bar x_t(\mu),\bar u_t) ) P_t(\lambda,\mu)\\
\qquad\qquad\qquad + \langle \sigma_{x}(\bar x_t(\lambda), \bar u_t), \sigma_{x}(\bar x_t(\mu), \bar u_t) \rangle_{L_2(\Xi,\mathbb{R})} P_t(\lambda,\mu) \\
\qquad\qquad\qquad + \langle \sigma_{x} (\bar x_t(\lambda),\bar u_t) + \sigma_{x} (\bar x_t(\mu),\bar u_t), Q_t(\lambda,\mu)\rangle_{L_2(\Xi,\mathbb{R})}\\
\qquad\qquad\qquad + \delta^{\ast}( l_{xx}(\bar x_t(\lambda), \bar u_t) ) + \delta^{\ast}( b_{xx}(\bar x_t(\lambda), \bar u_t) p_t(\lambda) )\\
\qquad\qquad\qquad + \delta^{\ast}( \langle \sigma_{xx}(\bar x_t(\lambda),\bar u_t), q_t \rangle_{L_2(\Xi,\mathbb{R})} ) ] \mathrm{d}t + Q_t(\lambda,\mu) \mathrm{d}W_t\\
P_T(\lambda,\mu) = \delta^{\ast}( h_{xx}(\bar x_T(\lambda)) ),
\end{cases}
\end{equation}
such that
\begin{equation}
\begin{split}
&\mathcal H(\bar x_t,v,p_t,q_t) -\mathcal H(\bar x_t,\bar u_t,p_t,q_t)\\
+ &\frac12 \Big \langle P_t(\lambda,\mu),\\
&\qquad \langle \sigma(\bar x_t(\lambda), v) - \sigma(\bar x_t(\lambda), \bar u_t) , \sigma(\bar x_t(\mu), v) - \sigma(\bar x_t(\mu), \bar u_t) \rangle_{L_2(\Xi,\mathbb{R})} \Big \rangle_{L^2(\Lambda^2)} \geq 0
\end{split}
\end{equation}
for all $v\in U$, and almost all $(t,\omega) \in [0,T]\times \Omega$. Here we denote by $\mathcal H$ the Hamiltonian, i.e. $\mathcal H: L^2(\Lambda) \times U \times L^2(\Lambda)\times L_2(\Xi,L^2(\Lambda)) \to \mathbb{R}$,
\begin{equation}
\mathcal H(x,u,p,q) := \int_{\Lambda} l(x(\lambda),u) \mathrm{d}\lambda + \langle p, b(x,u) \rangle_{L^2(\Lambda)} + \langle q, \sigma(x,u) \rangle_{L_2(\Xi,L^2(\Lambda))}.
\end{equation}
\end{theorem}

\begin{proof}
Equation \cref{variational1} states
\begin{equation}
\begin{split}
&\mathbb{E} \left [ \int_0^T \left \langle b(\bar x_t, u^{\varepsilon}_t) - b(\bar x_t, \bar u_t) , p_t \right \rangle_{L^2(\Lambda)} + \left \langle \sigma(\bar x_t, u^{\varepsilon}_t) - \sigma( \bar x_t, \bar u_t) , q_t \right \rangle_{L_2(\Xi,L^2(\Lambda))} \mathrm{d}t \right ]\\
&+ \frac12 \mathbb{E} \left [ \int_0^T \langle P^{\eta}_t, \Phi^{\varepsilon}_t \rangle_{L^2(\Lambda^2)} + \langle Q^{\eta}_t, \Psi^{\varepsilon}_t \rangle_{L_2(\Xi,L^2(\Lambda^2))} \mathrm{d}t \right ]\\
&+ \mathbb{E} \left [ \int_0^T \int_{\Lambda} l(\bar x_t(\lambda), u^{\varepsilon}_t) - l( \bar x_t(\lambda), \bar u_t) \mathrm{d}\lambda \mathrm{d}t \right ] \\
&+\frac12 \mathbb{E} \left [ \int_{\Lambda} h_{xx}( \bar x_T(\lambda)) y^{\varepsilon}_T(\lambda) y^{\varepsilon}_T(\lambda) \mathrm{d}\lambda - \int_{\Lambda^2} h^{\eta}_{xx}(\lambda,\mu) Y^{\varepsilon}_T(\lambda,\mu) \mathrm{d}\lambda \mathrm{d}\mu \right ] \geq o(\varepsilon).
\end{split}
\end{equation}
Localizing by dividing by $\varepsilon$ and taking the limit $\varepsilon\to 0$ yields 
\begin{equation}\label{limitepsilon}
\begin{split}
&\frac{1}{\varepsilon} \mathbb{E} \left [ \int_0^T \left \langle b(\bar x_t, u^{\varepsilon}_t) - b(\bar x_t, \bar u_t) , p_t \right \rangle_{L^2(\Lambda)} + \left \langle \sigma(\bar x_t, u^{\varepsilon}_t) - \sigma( \bar x_t, \bar u_t) , q_t \right \rangle_{L_2(\Xi,L^2(\Lambda))} \mathrm{d}t \right ]\\
&+ \! \frac{1}{2\varepsilon} \mathbb{E}\! \left [ \int_0^T \langle P^{\eta}_t, \Phi^{\varepsilon}_t \rangle_{L^2(\Lambda^2)} + \langle Q^{\eta}_t, \Psi^{\varepsilon}_t \rangle_{L_2(\Xi,L^2(\Lambda^2))} \mathrm{d}t \right ]\\
&+ \frac{1}{\varepsilon} \mathbb{E} \left [ \int_0^T \int_{\Lambda} l(\bar x_t(\lambda), u^{\varepsilon}_t) - l( \bar x_t(\lambda), \bar u_t) \mathrm{d}\lambda \mathrm{d}t \right ] \\
\to& \left \langle b(\bar x_{\tau}, v) - b(\bar x_{\tau}, \bar u_{\tau}) , p_{\tau} \right \rangle_{L^2(\Lambda)} + \left \langle \sigma(\bar x_{\tau}, v) - \sigma( \bar x_{\tau}, \bar u_{\tau}) , q_{\tau} \right \rangle_{L_2(\Xi,L^2(\Lambda))}\\
&+ \int_{\Lambda} l(\bar x_{\tau}(\lambda), v) - l( \bar x_{\tau}(\lambda), \bar u_{\tau}) \mathrm{d}\lambda\\
&+ \frac12 \Big \langle P^{\eta}_{\tau}(\lambda,\mu),\! \langle \sigma(\bar x_{\tau}(\lambda), v)\! -\! \sigma(\bar x_{\tau}(\lambda), \bar u_{\tau}) , \sigma(\bar x_{\tau}(\mu), v)\! -\! \sigma(\bar x_{\tau}(\mu), \bar u_{\tau}) \rangle_{L_2(\Xi,\mathbb{R})} \Big \rangle_{L^2(\Lambda^2)}\!.
\end{split}
\end{equation}
Notice that all but the remaining term in
\begin{equation}
\mathbb{E} \left [ \int_0^T \langle P^{\eta}_t, \Phi^{\varepsilon}_t \rangle_{L^2(\Lambda^2)} + \langle Q^{\eta}_t, \Psi^{\varepsilon}_t \rangle_{L_2(\Xi,L^2(\Lambda^2))} \mathrm{d}t \right ]
\end{equation}
are of order $o(\varepsilon)$. It remains to prove that
\begin{equation}
\lim_{\varepsilon\to 0} \frac{1}{\varepsilon} \mathbb{E} \left [ \int_{\Lambda} h_{xx}( \bar x_T(\lambda)) y^{\varepsilon}_T(\lambda) y^{\varepsilon}_T(\lambda) \mathrm{d}\lambda - \int_{\Lambda^2} h^{\eta}_{xx}(\lambda,\mu) Y^{\varepsilon}_T(\lambda,\mu) \mathrm{d}\lambda \mathrm{d}\mu \right ]
\end{equation}
vanishes in the limit $\eta\to 0$. Using \cref{subsequence} and the compact embedding $H^{\gamma}_0(\Lambda) \subset \subset L^2(\Lambda)$, $\gamma \in (0,1/2)$ (see e.g. \cite{demengel2012}, Theorem 4.54), we can extract a subsequence of $y^{\varepsilon}_T/\sqrt{\varepsilon}$ converging in $L^2(\Lambda)$ to some $\tilde y_T\in L^2(\Lambda)$. Therefore
\begin{equation}
\begin{split}
&\lim_{\varepsilon\to 0} \frac{1}{\varepsilon} \mathbb{E} \left [ \int_{\Lambda} h_{xx}( \bar x_T(\lambda)) y^{\varepsilon}_T(\lambda) y^{\varepsilon}_T(\lambda) \mathrm{d}\lambda - \int_{\Lambda^2} h^{\eta}_{xx}(\lambda,\mu) Y^{\varepsilon}_T(\lambda,\mu) \mathrm{d}\lambda \mathrm{d}\mu \right ]\\
=& \mathbb{E} \left [ \int_{\Lambda} h_{xx}( \bar x_T(\lambda)) \tilde y_T(\lambda) \tilde y_T(\lambda) \mathrm{d}\lambda - \int_{\Lambda^2} h^{\eta}_{xx}(\lambda,\mu) \tilde y_T(\lambda) \tilde y_T(\mu) \mathrm{d}\lambda \mathrm{d}\mu \right ],
\end{split}
\end{equation}
which vanishes in the limit $\eta \to 0$. This concludes the proof.
\end{proof}

\section*{Acknowledgement} This work has been funded by Deutsche Forschungsgemeinschaft (DFG) through grant CRC 910 ``Control of self-organizing nonlinear systems: Theoretical methods and concepts of application,'' Project (A10) ``Control of stochastic mean-field equations with applications to brain networks.'' The authors would like to thank the referees for their valuable feedback.



\end{document}